\newif\iffinal
\finalfalse	

\documentclass[letterpaper,12pt,reqno,dvipsnames]{amsart}
\RequirePackage[utf8]{inputenc}
\usepackage[portrait,margin=3cm]{geometry}
\usepackage{mathrsfs,soul}
\usepackage{hyperref}
\usepackage[foot]{amsaddr}
\usepackage{amssymb,amsthm,amsfonts,amsbsy,latexsym,dsfont}
\usepackage[textsize=small]{todonotes}       
\usepackage{graphicx}
\usepackage[numeric,initials,nobysame]{amsrefs}
\usepackage{upref,setspace}
\usepackage{xcolor,colortbl}
\usepackage{enumerate}
\newenvironment{enumeratei}{\begin{enumerate}[\upshape (i)]}{\end{enumerate}}

\usepackage{tikz}
\usetikzlibrary{calc,intersections,through,backgrounds,shapes.geometric}
\usetikzlibrary{graphs}
 	\usepackage{pgfplots}
\usepackage{paralist}

\usepackage{lipsum}

\newcounter{mycounter}

\numberwithin{equation}{section}
\numberwithin{figure}{section}
\numberwithin{table}{section}

\usepackage{caption}
\usepackage{subcaption}
\usepackage{color}

\usepackage[mathscr]{eucal}
\usepackage{graphicx}
\usepackage{sidecap}
\usepackage{latexsym}
\usepackage{psfrag}
\usepackage{epsfig}
\usepackage{enumerate}
\usepackage{amsmath}
\usepackage{amsthm}
\usepackage{amssymb}
\usepackage{pdfsync}
\usepackage{fourier}

\DeclareMathAlphabet{\mathpzc}{OT1}{pzc}{m}{it}

\usepackage{amsfonts}
\usepackage{color}
 \usepackage[textsize=small]{todonotes}
\usepackage{hyperref}

\newtheorem{theorem}{\bf Theorem}[section]

\newtheorem{corollary}[theorem]{\bf Corollary}
\newtheorem{lemma}[theorem]{\bf Lemma}

\newtheorem{condition}{\bf Condition}[section]
\newtheorem{proposition}[theorem]{\bf Proposition}

\theoremstyle{remark}
\newtheorem{example}{\bf Example}[section]
\newtheorem{remark}{\bf Remark}[section]

\newcounter{constant}
\numberwithin{equation}{section}
\numberwithin{theorem}{section}

\newcommand{\Erdos}{Erd\H{o}s-R\'enyi }

\newcommand{\one} {{\boldsymbol{1}}}
\newcommand{\half}{{\frac{1}{2}}}

\newcommand{\lan}{\langle}
\newcommand{\ran}{\rangle}

\newcommand{\Rd}{{\Rmb^d}}

\newcommand{\intRd}{\int_\Rd}

\newcommand{\Cmb}{{\mathbb{C}}}
\newcommand{\Dmb}{{\mathbb{D}}}

\newcommand{\Nmb}{{\mathbb{N}}}

\newcommand{\Pmb}{{\mathbb{P}}}
\newcommand{\Qmb}{{\mathbb{Q}}}
\newcommand{\Rmb}{{\mathbb{R}}}
\newcommand{\Smb}{{\mathbb{S}}}

\newcommand{\Bmc}{{\mathcal{B}}}
\newcommand{\Cmc}{{\mathcal{C}}}

\newcommand{\Fmc}{{\mathcal{F}}}
\newcommand{\Gmc}{{\mathcal{G}}}

\newcommand{\Lmc}{{\mathcal{L}}}
\newcommand{\Mmc}{{\mathcal{M}}}

\newcommand{\Pmc}{{\mathcal{P}}}

\newcommand{\Rmc}{{\mathcal{R}}}
\newcommand{\Smc}{{\mathcal{S}}}
\newcommand{\Tmc}{{\mathcal{T}}}
\newcommand{\Umc}{{\mathcal{U}}}

\newcommand{\Ebf}{{\mathbf{E}}}

\newcommand{\Ebd}{{\boldsymbol{E}}}

\newcommand{\Kbd}{{\boldsymbol{K}}}

\newcommand{\Nbd}{{\boldsymbol{N}}}

\newcommand{\Nalpha}{{\boldsymbol{N_\alpha}}}

\newcommand{\Ngamma}{{\boldsymbol{N_\gamma}}}

\newcommand{\pbd}{{\boldsymbol{p}}}
\newcommand{\Pbd}{{\boldsymbol{P}}}
\newcommand{\phibd}{{\boldsymbol{\phi}}}

\newcommand{\psibd}{{\boldsymbol{\psi}}}

\newcommand{\bbar}{{\bar{b}}}

\newcommand{\mubar}{{\bar{\mu}}}

\newcommand{\Nbar}{{\bar{N}}}

\newcommand{\omegabar}{{\bar{\omega}}}

\newcommand{\pbar}{{\bar{p}}}

\newcommand{\sigmabar}{{\bar{\sigma}}}

\newcommand{\Tmchat}{{\hat{\Tmc}}}

\newcommand{\etatil}{{\tilde{\eta}}}

\newcommand{\Fmctil}{{\tilde{\Fmc}}}

\newcommand{\Jtil}{{\tilde{J}}}

\newcommand{\mutil}{{\tilde{\mu}}}

\newcommand{\Tmctil}{{\tilde{\Tmc}}}

\newcommand{\EbfP}{{\Ebf_{\Pmb^N}}}

\newcommand{\EbfPV}{{\Ebf_{\Pmb^N,V}}}
\newcommand{\ialpha}{{i_\alpha}}
\newcommand{\igamma}{{i_\gamma}}

\newcommand{\ZiNt}{{Z_t^{i,N}}}
\newcommand{\ZjNt}{{Z_t^{j,N}}}
\newcommand{\ZiNs}{{Z_s^{i,N}}}
\newcommand{\ZjNs}{{Z_s^{j,N}}}

\newcommand{\set}[1]{\left\{#1\right\}}

\newcommand{\erdos}{Erd\H{o}s-R\'enyi }

\begin{document}

\title[WIPS on random graphs]{Weakly interacting particle systems on inhomogeneous random graphs}

\date{}
\subjclass[2010]{Primary: 60C05, 05C80, 60F05, 60K35, 60H30, 60J70. }
\keywords{inhomogeneous random graphs, dynamical random graphs, weakly interacting diffusions, propagation of chaos, central limit theorems, multi-type populations, interacting particle systems. }

\author[Bhamidi]{Shankar Bhamidi$^1$}
\address{$^1$Department of Statistics and Operations Research, 304 Hanes Hall, University of North Carolina, Chapel Hill, NC 27599}
\author[Budhiraja]{Amarjit Budhiraja$^1$}
\author[Wu]{Ruoyu Wu$^2$}
\address{$^2$Division of Applied Mathematics, 182 George Street, Brown University, Providence, RI 02912}
\email{bhamidi@email.unc.edu,  budhiraj@email.unc.edu, ruoyu\_wu@brown.edu}
\maketitle
\begin{abstract}
	We consider weakly interacting diffusions on time varying random graphs. The system consists of a large number of nodes in which the state of each node is governed by
	a diffusion process that is influenced by the neighboring nodes. The collection of neighbors of a given node changes dynamically over time and is determined through a time evolving random graph process.
	A law of large numbers and a propagation of chaos result is established for a multi-type population setting where at each instant the interaction between nodes is given by an inhomogeneous random graph
	which may change over time. This result covers the setting in which the edge probabilities between any two nodes is allowed to decay to $0$ as the size of the system grows. A central limit theorem is established for the single-type population case under stronger conditions on the edge probability function.
\end{abstract}


\section{Introduction} \label{sec:intro}

In this work we study some asymptotic results for large particle systems given as weakly interacting diffusion processes on time varying inhomogeneous random graphs. 
The model is described in terms of two types of stochastic dynamical systems, one that describes the evolution of 
the  graph that governs the interaction between nodes of the system ({\em dynamics of the network}) and the other that describes the evolution of the states of all the nodes in the system ({\em dynamics on the network}). 
We consider a setting where the interaction between the nodes is {\em weak} in that the `strength' of the interaction between a node and its neighbor is inversely proportional to the total number of neighbors of that node. 
Such stochastic systems arise in many different areas, such as social networks (e.g.\ in the study of gossip algorithms \cite{bertsekas1989parallel,shah2009gossip}), biological systems (e.g.\ swarming and flocking models, see \cite{bertozzi1} and references therein), neurosciences (e.g.\ in modeling of networks of spiking neurons, see \cite{BaladronFaugeras2012,Bossy2015clarification} and references therein), and mathematical finance (e.g.\ in modeling correlations between default probabilities of multiple firms \cite{CvitanicMaZhang2012law}). 

The case where the interaction graph is complete, i.e.\ each node interacts with every other node, is classical and dates back to works of Boltzmann, Vlasov, McKean and others (see \cite{Sznitman1991,Kolokoltsov2010} and references therein). Original motivation for the study of such systems came from Statistical Physics but, as noted above, in recent years similar models have arisen in many different application areas, ranging from mathematical finance and chemical and biological systems to communication networks and social sciences (see \cite{BudhirajaDupuisFischerRamanan2015limits} for an extensive list of references).
The asymptotic picture in the setting of a complete graph is well resolved and many different results have been established, including laws of large numbers (LLN), propagation of chaos (POC) properties, and central limit theorems (CLT), see e.g.\ \cite{McKean1966class,McKean1967propagation,BraunHepp1977vlasov,Dawson1983critical,Tanaka1984limit,Oelschlager1984martingale,Sznitman1984,Sznitman1991,GrahamMeleard1997,ShigaTanaka1985,Meleard1998}.
 A number of variations have also been studied.
For example, in~\cite{KurtzXiong1999,KurtzXiong2004,BudhirajaSaha2014,BudhirajaWu2015} a setting where a common noise process  influences the dynamics of every particle is considered and limit theorems of the above form are established.
For a setting with $K$-different subpopulations within each of which particle evolution is exchangeable, LLN and POC have been studied in~\cite{BaladronFaugeras2012}, and a corresponding CLT has been established in~\cite{BudhirajaWu2015}.
Mean field results for heterogeneous populations have also been  studied in ~\cite{Collet2014macroscopic,ChongKluppelberg2015partial}. All of these  papers consider  complete interaction graphs where every vertex influences every other vertex (albeit in a `weak' fashion).

\subsection{Informal overview of our contributions and proof techniques}
\label{sec:inf-over}

The goal of the current work is to develop an analogous limit theory when the interaction graph is not complete and is possibly time varying. 
The two main results of this work are Theorem \ref{thm:NpN_rate}
and Theorem \ref{thm:CLT}.  Theorem \ref{thm:NpN_rate} and its corollaries (Corollary \ref{cor:poc} and Corollary \ref{cor:LLN}) give a law of large numbers result and a propagation of chaos property whereas Theorem
\ref{thm:CLT} proves a central limit theorem.

 In the area of interacting particle systems there is a large amount of work for many different settings with dynamics on random graphs, for example, the voter model or the contact process \cite{aldous2013interacting,liggett2012interacting,durrett1995ten,durrett-book}. There is much less work for the class of models considered in this paper. The closest in spirit to the work here is \cite{bertozzi1,bertozzi2} where the authors consider an \erdos random graph $\Gmc_N(p) = \set{\xi_{ij}: 1\le i< j \le N }$, where $\xi_{ij}=1$ if there is an edge between vertices $i$ and $j$, and $0$ otherwise.  Using the interaction structure generated by this graph these papers consider a  family of coupled ordinary differential equations (ODE),
\[\frac{d z_i}{dt} = \frac{1}{N} \sum_{j=1}^N \xi_{ij} F(|z_i - z_j|) \frac{z_i - z_j}{|z_i - z_j|}, \qquad i=1,\ldots, N. \]
Here $F(\cdot)$ is a suitable repulsive-attractive force that captures the phenomenon that particles are attracted to each other unless they get too close, in which case  they are repelled.
The aim of these works is  to understand the consensus behavior of the associated coupled system of ODE as $N\to\infty$. The model considered in the current work is different in several aspects: (1) The random graph is allowed to change over time in quite a general fashion; (2) the population can be multi-type, namely the associated random graph can be `inhomogeneous'; (3) Edge probabilities can change with system size and are even allowed to decay to zero (for the LLN result) as system size increases; (4) the dynamics of the node states can have external noise and be described through stochastic differential equations (SDE) rather than ODE. 
Another recent work that is close to ours is \cite{Delattre2016} where interacting diffusions on static graphs are studied and quenched results on propagation of chaos are established.

For the law of large numbers result, we consider a random graph model that is a time evolving version of the inhomogeneous random graph models studied by Bollobas, Janson and Riordan in \cite{bollobas-riordan-janson}.
Roughly speaking, each node can be of $K$ possible types and at any time instant edges between two nodes form independently with probabilities depending only on the types of the two nodes.
These probabilities may change over time and are allowed to decay with $N$ (see Condition \ref{cond:cond1}). The evolution of the node states is described through a collection of weakly interacting stochastic
differential equations(SDE) such that the interaction of a particular particle with its neighbors is given through the coefficients in the SDE for that particle that depend on the states of all its neighbors through functions that depend only on the node-type (see equation \eqref{eq:ZiNt}
for a precise description of the evolution). The proof of the law of large numbers relies on certain coupling arguments along with exchangeability properties of nodes of each type, and concentration inequalities for various functionals of the random graph.

Our second main result proves a central limit theorem for scaled functionals of the empirical measures of the particle states. For simplicity, here we consider the single type setting, namely $K=1$. 
The proof of CLT relies on a change of measure technique using Girsanov's theorem, which goes back to \cite{Sznitman1984,ShigaTanaka1985}.
This technique reduces the problem to a setting with i.i.d.\ particles and edges, while the price to pay is that one must carefully analyze the asymptotic behavior of the Radon-Nikodym derivative $J^N(T)$, which is given in terms of quantities (see $J^{N,1}(T)$ and $J^{N,2}(T)$ in \eqref{eq:JN1} and \eqref{eq:JN2} for the martingale and the quadratic variation part, respectively) involving all particles and edges.
For the asymptotics of the quadratic variation part $J^{N,2}(T)$, one can modify the system to reduce the analysis to the setting of a complete graph by carefully studying the effect of the time evolving random interaction graph (Lemma \ref{lem:Ttil}) and bounding the corresponding error (Lemma \ref{lem:tiljn2}), and then apply classical limit theorems for symmetric statistics \cite{Dynkin1983}.
For the martingale part $J^{N,1}(T)$, however, one cannot easily modify the system to replace the random graph  by a complete graph while keeping the  particle interaction to be  i.i.d.
Due to the interaction being governed by a time evolving random graph, neither can one  apply classical techniques from \cite{Sznitman1984,ShigaTanaka1985}, since the entire collection of particles and edges cannot be viewed as a collection of i.i.d.\ particle-edge combinations. 
The asymptotic analysis of $J^{N,1}(T)$ presents one of the main challenges in the proof.
For this, we first show in Lemma \ref{lem:JN1_JN1_tilde} that $J^{N,1}(T)$ is asymptotically close to $\Jtil^{N,1}(T)$ defined in \eqref{eq:JN1til}, which is written in a form similar to an incomplete U-statistic and that has `lesser dependence' on the random interaction graph. The significance of this result is that incomplete U-statistics, first introduced by Blom \cite{Blom1976some}, have been shown to be asymptotically normal by  Janson \cite{Janson1984}, under suitable conditions. We extend Janson's limit result to the setting of stochastic processes where the incompleteness shows up in the integrands of certain stochastic integrals  and combine it with classical limit theorems for symmetric statistics, to obtain  in Lemma \ref{lem:key}, a key characterization of the asymptotic distribution of $\Jtil^{N,1}(T)$ as the sum of a normal random variable and a multiple Wiener integral (see Section \ref{sec:asymp_symmetric_statistics} for precise definitions) with certain independence properties.
These properties allow us to obtain joint asymptotic behavior of $J^{N,1}(T)$ and $J^{N,2}(T)$, which completes the analysis of the Radon-Nikodym derivative.

The proof of the central limit theorem requires a stronger condition on the edge probabilities (see Condition \ref{cond:cond2}) than that needed for the law of large numbers. In particular, here we are unable to treat the case where the edge probabilities decay to $0$ with $N$.
One interesting aspect of the proof is the role of the non-degeneracy assumption on edge probabilities in Condition \ref{cond:cond2}. For simplicity, consider the setting where the random graph is static, given at time instant $0$ according to an \Erdos random graph with edge probabilities $p_N$ and suppose that $p_N \to p>0$ as $N \to \infty$. Then, as seen in Theorem \ref{thm:CLT}, the variance of the limiting Gaussian random field does not depend on $p$. One may conjecture that because of this fact one may be able to relax the condition $p>0$ and allow $p_N$ to converge to $0$ at an appropriate rate. However, as noted earlier, a crucial ingredient in our proof is the study of the asymptotic behavior of the Radon-Nikodym derivative $J^N(T)$. The limit of this random variable is described (see Proposition \ref{prop:key_joint_cvg})
in terms of a Normal random variable $Z$ (along with a multiple Wiener integral of order $2$) with mean $b(p)$ and variance $\sigma^2(p)$ where both $b(p)$ and $\sigma^2(p)$ approach $\infty$ as $p \to 0$.
The arguments of Section \ref{sec:complete_pf_CLT} show that the distribution of the limiting Gaussian random field depend on the quantity $\Ebf e^Z = e^{b(p)+ \sigma^2(p)/2}$ and one finds that although both
$b(p)$ and $\sigma^2(p)$ diverge, the quantity $b(p)+ \sigma^2(p)/2 =0$ for each $p>0$. This is the key observation in the proof and the reason the limit random field does not depend on $p$ . However, when $p_N \to 0$, a similar analysis of the asymptotics of the Radon-Nikodym derivative $J^N(T)$ cannot be carried out and even the tightness of this term is unclear. Proving a suitable fluctuations result in this regime (i.e.\ when $p_N \to 0$ at a suitable rate) is an interesting and challenging open problem.


\subsection{Organization}
\label{sec:org}

The paper is organized as follows. We conclude this Section by outlining the notation used in the rest of the paper. 
Our model of weakly interacting multi-type diffusions on random graphs is introduced in Section \ref{sec:mod}. 
In Section \ref{sec:laws-large-results} two basic conditions (Conditions \ref{cond:coefficients} and \ref{cond:cond1}) on coefficients in the model and sparsity of the interaction graph are stated, under which a law of large numbers and propagation of chaos property are established in Theorem \ref{thm:NpN_rate} and its corollaries.
Next in Section \ref{sec:fluctuations} we present a central limit theorem (Theorem \ref{thm:CLT}) in the single-type setting  under a stronger condition (Condition \ref{cond:cond2}) on sparsity of the interaction graph.
The rest of this paper gives proofs of Theorems \ref{thm:NpN_rate} and \ref{thm:CLT}.
We start in Section \ref{sec:pre_results} with some preliminary concentration results on the degree distribution.
Theorem \ref{thm:NpN_rate} is proved in Section \ref{sec:pf_NpN_rate}.
In Section \ref{sec:pf_CLT} we give the proof of Theorem \ref{thm:CLT}.
The proofs of several technical Lemmas and auxiliary results are given in the Appendix.

\subsection{Notation}
\label{sec:not}
The following notation will be used in the sequel. 
For a Polish space $(\Smb, d(\cdot,\cdot))$, denote the corresponding Borel $\sigma$-field by $\Bmc(\Smb)$.
For a signed measure $\mu$ on $\Smb$ and $\mu$-integrable function $f \colon \Smb \to \Rmb$, let $\langle f,\mu \rangle \doteq \int f \, d\mu$.
Denote by $\Pmc(\Smb)$ (resp.\ $\Mmc(\Smb)$) the space of probability measures (resp.\ sub-probability measures) on $\Smb$, equipped with the topology of weak convergence.
A convenient metric for this topology is the bounded-Lipschitz metric $d_{BL}$, defined as
\begin{equation*}
d_{BL}(\nu_1,\nu_2) \doteq \sup_{\|f\|_{BL} \le 1} | \langle f, \nu_1 - \nu_2 \rangle |, \quad \nu_1, \nu_2 \in \Mmc(\Smb),
\end{equation*}
where $\|\cdot\|_{BL}$ is the bounded Lipschitz norm, i.e.\ for $f \colon \Smb \to \Rmb$,
\begin{equation*}
	\|f\|_{BL} \doteq \max \{\|f\|_\infty, \|f\|_L\}, \quad \|f\|_\infty \doteq \sup_{x \in \Smb} |f(x)|, \quad \|f\|_L \doteq \sup_{x \ne y} \frac{|f(x)-f(y)|}{d(x,y)}.
\end{equation*}
Denote by $\Cmb_b(\Smb)$ the space of real bounded and continuous functions.
For a measure $\nu$ on $\Smb$ and a Hilbert space $H$, let $L^2(\Smb,\nu,H)$ denote the space of measurable functions $f \colon \Smb \to H$ such that $\int_\Smb \| f(x) \|^2_H \, \nu(dx) < \infty$, where $\| \cdot \|_H$ is the norm on $H$.
When $H = \Rmb$, we write $L^2(\Smb,\nu)$ or simply $L^2(\nu)$ if  $\Smb$ is clear from context.

Fix $T < \infty$. 
All stochastic processes will be considered over the time horizon $[0,T]$. 
We will use the notations $\{X_t\}$ and $\{X(t)\}$ interchangeably for stochastic processes.
For a Polish space $\Smb$, denote by $\Cmb([0,T]:\Smb)$ (resp.\ $\Dmb([0,T]:\Smb)$) the space of continuous functions (resp.\ right continuous functions with left limits) from $[0,T]$ to $\Smb$, endowed with the uniform topology (resp.\ Skorokhod topology). 
For $d \in \Nmb$, let $\Cmc_d \doteq \Cmb([0,T]:\Rmb^d)$ and $\|f\|_{*,t} \doteq \sup_{0 \le s \le t} \|f(s)\|$ for $f \in \Cmc_d$, $t \in [0,T]$.
We say a collection $\{ X^m \}$ of $\Smb$-valued random variables is tight if the distributions of $X^m$ are tight in $\Pmc(\Smb)$.
We use the symbol `$\Rightarrow$' to denote convergence in distribution.
The distribution of an $\Smb$-valued random variable $X$ will be denoted as $\Lmc(X)$. 
Expected value under a probability distribution $\Pmb$ will be denoted as $\Ebf_{\Pmb}$ but when clear from the context,
$\Pmb$ will be suppressed from the notation.


We will usually denote by $\kappa, \kappa_1, \kappa_2, \dotsc$, the constants that appear in various estimates within a proof. 
The value of these constants may change from one proof to another.

\section{Model}
\label{sec:mod}

We now describe the precise model that will be studied in this work. Law of large numbers will be established in the general $K$-type setting described below whereas  for the central limit theorem in Section \ref{sec:fluctuations} we will consider for simplicity the case $K=1$.
 
\subsection{Random graph}
Our random graph model is a (possibly time evolving) version in the class of inhomogeneous random graph models studied by Bollobas, Janson and Riordan in \cite{bollobas-riordan-janson}. 
We start with $N$ vertices represented by the vertex set $\Nbd \doteq \set{1,2,\ldots, N}$. 
Assume that each vertex can be one of $K$ possible types (sometimes called populations) labelled according to $\Kbd \doteq \set{1,2,\ldots, K}$. We will require, for every $\alpha \in  \Kbd$, the number of type $\alpha$ vertices, denoted by $N_{\alpha}$, to approach infinity and the ratio $N_{\alpha}/N$ to converge to  a  positive value as $N\to \infty$.
It will be notationally convenient to assume that $N_{\alpha}$ are nondecreasing in $N$ so that by reindexing if needed the type $\alpha$ vertices in the $N$-th model can be described through a membership map $\pbd: \Nmb \to \Kbd$ such that
type $\alpha$-vertices in the $N$-th model are given as $\Nalpha \doteq \{i \in \Nbd: \pbd(i)=\alpha\}$.
In this work we take the membership map to be fixed over time. It may be viewed as a given deterministic function or determined through a sample realization (which is then fixed throughout) of $N$ i.i.d.\ 
$\Kbd$-valued random variables
$\{p(i)\}_{i=1}^N$ distributed according to some probability distribution ${\boldsymbol\pi} = (\pi(\alpha), \alpha \in \Kbd)$ on $\Kbd$.  The evolution of the random graph is given in terms of a stochastic process
$\set{\xi_{ij}^N(t): t\in [0,T], 1\leq i\leq j\leq N}$ given on some filtered probability space $(\Omega,\Fmc,\Pbd,\{\Fmc_t\})$. We assume that 
%
%
%
%
%
%
$\xi_{ii}^N(t) \equiv 1$ for all $t\in [0,T]$ while $\{\xi_{ij}^N(t) \equiv \xi_{ji}^N(t): 1 \le i < j \le N\}$ are mutually independent $\{\Fmc_t\}$-adapted RCLL(right continuous with left limits) processes such that
\begin{equation*}
	\Pbd(\xi_{ij}^N(t)=1) = 1 - \Pbd(\xi_{ij}^N(t)=0) = p_{\alpha\gamma,N}(t), \quad \alpha, \gamma \in \Kbd, i \in \Nalpha, j \in \Ngamma, i \ne j,
\end{equation*}
where $p_{\alpha\gamma,N}$ is a $[0,1]$-valued continuous function on $[0,T]$ for every $\alpha, \gamma \in \Kbd$, $N \in \mathbb{N}$.
Note that $p_{\alpha\gamma,N}=p_{\gamma\alpha,N}$ for all $\alpha,\gamma \in \Kbd$.
Given such a stochastic process, for any $t\in [0,T]$  the graph $\Gmc_N(t)$ is formed via the following procedure:  for any unordered pair $i\neq j \in \Nbd$ if $\xi_{ij}^N(t) =1$, this implies at that given time instant $t$, there exists an edge between vertex $i$ and $j$; if $\xi_{ij}^N(t) =0$ then no edge exists at time $t$. The self-loop edges $\xi_{ii}^N(\cdot) \equiv 1$ are used just as a simplification in the interacting particle system process described in the next subsection. 
Ignoring these self-loop edges, the random graph $\Gmc_N(t)$ belongs to the class of inhomogeneous random graph models analyzed in \cite{bollobas-riordan-janson} wherein edges are formed randomly between vertices and connection probabilities depend only on the `type' of the vertex. For later use, let $N_{i,\gamma}(t) \doteq \sum_{j \in \Ngamma} \xi_{ij}^N(t)$ for $i \in \Nbd$, $\gamma \in \Kbd$. 
Thus $N_{i,\gamma}(t)$ represents the number of $\gamma$-th type neighbors of vertex $i$ at time $t$. 
Note that $N_{i,\alpha}(t) \ge 1$ for all $t \in [0,T]$ and $i \in \Nalpha$. 
Define,
\begin{equation}
	\label{eqn:def-pbar}
	\pbar_N\doteq \inf_{t \in [0,T]} \min_{\alpha,\gamma \in \Kbd} p_{\alpha\gamma,N}(t).
\end{equation}
The following are two natural families for the edge process. 

\begin{example}[{\bf Static networks}] \label{eg:eg1}
	Let $\{\xi_{ij}^N(t)\}$ be unchanging over time, namely $\xi_{ij}^N(t) \equiv \xi_{ij}^N(0)$.
	In this case 
	\begin{equation*}
		\pbar_N = \min_{\alpha,\gamma \in \Kbd} p_{\alpha\gamma,N}(0), \quad N_{i,\gamma}(t) \equiv N_{i,\gamma}(0)
	\end{equation*} 
	for all $t \in [0,T]$, $i,j \in \Nbd$ and $\gamma \in \Kbd$. 
\end{example}

\begin{example}[{\bf Markovian  edge formation}] \label{eg:eg2}
For each $N\geq 1$ let $\{\xi^N_{ij}(0)\}_{1\le i < j \le N}$ be mutually independent $\{0,1\}$-valued random variables
such that $P( \xi_{ij}^N(0) = 1 ) = p_{\alpha\gamma,N}(0)$ for $i\neq j\in \Nbd$ with $i\in \Nalpha$ and $j\in \Ngamma$
and $\alpha, \gamma \in \Kbd$. Fix positive $\{\lambda_{\alpha\gamma,N} = \lambda_{\gamma\alpha,N}\}_{\alpha,\gamma \in \Kbd}$ and $\{\mu_{\alpha\gamma,N} = \mu_{\gamma\alpha,N}\}_{\alpha,\gamma \in \Kbd}$.	
For any two vertices $i\neq j\in \Nbd$ with $i\in \Nalpha$ and $j\in \Ngamma$, let 
$\set{\xi_{ij}^N(t): t\geq 0}$ be a $\{0,1\}$-valued Markov process with rate matrix,
	\begin{equation*}
		\Gamma_{\alpha\gamma,N} \doteq 
		\begin{bmatrix}
			-\lambda_{\alpha\gamma,N}	& \lambda_{\alpha\gamma,N} \\
			\mu_{\alpha\gamma,N}		& -\mu_{\alpha\gamma,N}
		\end{bmatrix}.
	\end{equation*}
	We assume that the evolutions of Markov chains for different edges are independent. In this setting
	\begin{equation*}
		p_{\alpha\gamma,N}(t) = p_{\alpha\gamma,N}(0) e^{-(\lambda_{\alpha\gamma,N}+\mu_{\alpha\gamma,N})t} + \frac{\lambda_{\alpha\gamma,N}}{\lambda_{\alpha\gamma,N} + \mu_{\alpha\gamma,N}} \left( 1-e^{-(\lambda_{\alpha\gamma,N}+\mu_{\alpha\gamma,N})t} \right)
	\end{equation*}
	and hence 
	\begin{equation*}
		\pushQED{\qed} \pbar_N = \min_{\alpha,\gamma \in \Kbd} \min \{ p_{\alpha\gamma,N}(0), p_{\alpha\gamma,N}(T) \} \ge \min_{\alpha,\gamma \in \Kbd} \min \Big\{ p_{\alpha\gamma,N}(0), \frac{\lambda_{\alpha\gamma,N}}{\lambda_{\alpha\gamma,N} + \mu_{\alpha\gamma,N}} \Big\}. 
	\end{equation*}
\end{example}
One can also allow non-Markovian edge formation processes where the holding times have general probability distributions that
satisfy appropriate conditions.

\subsection{Interacting particle system}
The main object of interest in this paper is a collection of $\Rd$-valued diffusion processes $\{Z^{1,N},\dotsc,Z^{N,N}\}$, representing trajectories of $N$  particles of $K$ types and that interact through the evolving graphical structure represented by $\set{\Gmc_N(t):t\geq 0}$. 
The dynamics is given in terms of a collection of stochastic differential equations (SDE) driven by mutually independent Brownian motions with each particle's initial condition governed independently by a probability law that depends only on its type.
The interaction between particles occurs through the coefficients of the SDE in that for the $i$-th particle $Z^{i,N}$, with $\pbd(i)=\alpha$, the coefficients depend on not only the $i$-th particle's current state, but also $K$ empirical measures of its neighbors corresponding to $K$ types. 
A precise formulation is as follows. Recall the filtered probability space $(\Omega,\Fmc,\Pbd,\{\Fmc_t\})$ on which the edge processes $\set{\xi_{ij}^N: i\leq j \in \Nbd}$ are given. 
We suppose that on this space we are also given an infinite collection of standard $d$-dimensional $\{\Fmc_t\}$-Brownian motions $\{W^i : i \in \Nbd\}$ and $\Fmc_0$-measurable $\Rd$-valued random variables $\{X^i_0 : i \in \Nbd \}$, with $\Lmc(X^i_0) = \mu^\alpha_0$ for $i \in \Nalpha, \alpha \in \Kbd$, such that  $\{W^i, X^i_0, \xi_{ij}^N: i\leq j \in \Nbd\}$ are mutually independent.

Recall that $N_{i,\gamma}(t) \doteq \sum_{j \in \Nbd_\gamma} \xi_{ij}^N(t)$ denotes the number of $\gamma$-th type particles that interact with $i$-th particle at time instant $t$. For $\alpha \in \Kbd$ and $i \in \Nalpha$ consider the collection of stochastic differential equations given by,
\begin{align}
	\ZiNt & = X^i_0 + \sum_{\gamma=1}^K \int_0^t b_{\alpha\gamma}(\ZiNs,\mu^{i,\gamma,N}_s) \, ds + \sum_{\gamma=1}^K \int_0^t \sigma_{\alpha\gamma}(\ZiNs,\mu^{i,\gamma,N}_s) \, dW^i_s, \label{eq:ZiNt} \\
	\mu^{i,\gamma,N}_t & = \frac{1}{N_{i,\gamma}(t)} \sum_{j \in \Ngamma} \xi_{ij}^N(t) \delta_{\ZjNt} \one_{\{N_{i,\gamma}(t) > 0\}},  
\end{align}
where for $\alpha,\gamma \in \Kbd$, $x \in \Rmb^d$ and $\theta \in \Mmc(\Rd)$, 
\begin{equation*}
	b_{\alpha\gamma}(x,\theta) \doteq \intRd \bbar_{\alpha\gamma}(x,y) \, \theta(dy), \quad \sigma_{\alpha\gamma}(x,\theta) \doteq \intRd \sigmabar_{\alpha\gamma}(x,y) \, \theta(dy)
\end{equation*} 
for suitable functions $\bbar_{\alpha\gamma} \colon \Rd \times \Rd \to \Rd$ and $\sigmabar_{\alpha\gamma} \colon \Rd \times \Rd \to \Rmb^{d \times d}$.
Under Condition \ref{cond:coefficients} one can easily establish  existence and uniqueness of pathwise solutions to the above system of stochastic differential equations  \cite{Sznitman1991}.

We can now summarize the main contributions of this work.
\begin{enumeratei}
	\item In Theorem \ref{thm:NpN_rate} and its corollaries we show that with suitable assumptions on coefficients (Condition \ref{cond:coefficients}) and a sparsity condition on the interaction graph that is formulated in terms of the decay rate of $\pbar_N$ (Condition \ref{cond:cond1}), a law of large numbers and propagation of chaos result hold.
\item In Section \ref{sec:fluctuations} we study the fluctuations of $\{Z^{i,N}\}$ from its law of large numbers limit by establishing a central limit theorem.
For simplicity, we study the single-type setting, i.e.\ $K=1$. 
Specifically, let
\begin{equation*}
	\eta^N(\phi) \doteq \frac{1}{\sqrt{N}} \sum_{i=1}^N \phi(Z^{i,N}), \quad \phi \in L^2_c(\Cmc_d,\mu),
\end{equation*}
where $L^2_c(\Cmc_d,\mu)$ is a family of functions on the path space that are suitably centered and have appropriate integrability properties (see Section \ref{sec:CLT} for definitions).
We show in Theorem \ref{thm:CLT} that under Condition \ref{cond:coefficients} and a stronger assumption on edge probability $p_N$ (Condition \ref{cond:cond2}) the family $\{\eta^N(\phi) : \phi \in L^2_c(\Cmc_d,\mu) \}$ converges weakly to a mean $0$ Gaussian field $\{\eta(\phi) : \phi \in L^2_c(\Cmc_d,\mu) \}$ in the sense of convergence of finite dimensional distributions.
\end{enumeratei}

\section{Laws of large numbers}
\label{sec:laws-large-results}

We now describe our main results. This section deals with the law of large numbers while the next section concerns central limit theorems. 
Recall the collection of  SDE $\{Z_t^{i, N}: i\in \Nbd\}$ describing the evolution of $N$ interacting particles, defined in Section \ref{sec:mod} via \eqref{eq:ZiNt}. Along with this system, we will also consider a related infinite system of equations for $\Rd$-valued nonlinear diffusions $X^i$, $i \in \Nmb$, given on $(\Omega,\Fmc,\Pbd,\{\Fmc_t\})$.
Let $\Nmb_\alpha \doteq \{i \in \Nmb: \pbd(i) = \alpha \}$.
For $\alpha \in \Kbd$ and $i \in \Nmb_\alpha$,
\begin{equation} \label{eq:Xit}
	X^i_t = X^i_0 + \sum_{\gamma=1}^K \int_0^t b_{\alpha\gamma}(X^i_s,\mu^\gamma_s) \, ds + \sum_{\gamma=1}^K \int_0^t \sigma_{\alpha\gamma}(X^i_s,\mu^\gamma_s) \, dW^i_s, \quad \mu^\gamma_t = \Lmc(X^j_t), \quad i \in \Nmb_\alpha, j \in \Nmb_\gamma.
\end{equation}
The existence and uniqueness of pathwise solutions of \eqref{eq:ZiNt} and \eqref{eq:Xit} can be shown under the following conditions on the coefficients (cf.\ \cite{Sznitman1991}).

\begin{condition} \label{cond:coefficients}
	There exists some $L \in (0,\infty)$ such that for all $\alpha,\gamma \in \Kbd$, $\|\bbar_{\alpha\gamma}\|_{BL} \le L$ and $\|\sigmabar_{\alpha\gamma}\|_{BL} \le L$.
\end{condition}

We start with the following moment estimate. 
The proof is given in Section~\ref{sec:pf_NpN_rate}. For later use define $\Nbar \doteq \min_{\alpha \in \Kbd} N_\alpha$.

\begin{theorem} \label{thm:NpN_rate}
	Suppose Condition~\ref{cond:coefficients} holds.
	Then
	\begin{equation} \label{eq:NpN rate}
		\sup_{N \ge 1} \max_{i \in \Nbd} \sqrt{\Nbar \pbar_N} \Ebd \|Z^{i,N} - X^i\|_{*,T}^2 < \infty.
	\end{equation}
\end{theorem}

\begin{remark}
	From the above theorem it is clear that $\Ebd \|Z^{i,N} - X^i\|_{*,T}^2$ is of order at most $(\Nbar \pbar_N)^{-1/2}$.
	If one assumes diffusion coefficients $\{\sigma_{\alpha\gamma}\}_{\alpha,\gamma \in \Kbd}$ to be constants, then the following result with a better order can be obtained:
	\begin{equation}
		\label{eq:NpN_better}
		\sup_{N \ge 1} \max_{i \in \Nbd} \sqrt{\Nbar \pbar_N} \Ebd \|Z^{i,N} - X^i\|_{*,T} < \infty.
	\end{equation}	
	See Remark \ref{rem:remsec6} for comments on this point.
\end{remark}

We will make the following assumption on $\pbar_N$. 
\begin{condition} \label{cond:cond1}
	$\Nbar \pbar_N \to \infty$ as $N \to \infty$.
\end{condition}

Theorem~\ref{thm:NpN_rate} together with a standard argument (cf.\ \cite{Sznitman1991}) implies that, under Conditions~\ref{cond:coefficients} and~\ref{cond:cond1}, the following propagation of chaos result holds.
We omit the proof.

\begin{corollary} \label{cor:poc}
	Suppose Conditions~\ref{cond:coefficients} and~\ref{cond:cond1} hold.
	Then for any $n$-tuple  $(i_1^N,\dotsc,i_n^N) \in \Nbd^n$ with $i_j^N \neq i_k^N$ whenever $j\neq k$, and
	$\pbd(i_j^N) = \alpha_j$, $j = 1,\dotsc,n$, 
	\begin{equation*}
		\Lmc(\{Z^{i_1^N,N},\dotsc,Z^{i_n^N,N}\}) \to \mu^{\alpha_1} \otimes \dotsb \otimes \mu^{\alpha_n}
	\end{equation*}
	in $\Pmc(\Cmc_d^n)$ as $N \to \infty$, where $\mu^\alpha \doteq \Lmc(X^i) \in \Pmc(\Cmc_d)$ for $\alpha \in \Kbd$ and $i \in \Nmb_\alpha$.
\end{corollary}

Using above results and an argument similar to~\cite{Sznitman1991} one can further show the following law of large numbers result. 
Proof is included in Appendix~\ref{sec:pf_Cor_LLN} for completeness.

\begin{corollary} \label{cor:LLN}
	Suppose Conditions~\ref{cond:coefficients} and~\ref{cond:cond1} hold.
	Then \\
	(a) For each $\alpha \in \Kbd$, as $N \to \infty$, $$\mu^{\alpha,N} \doteq \frac{1}{N_\alpha} \sum_{i \in \Nalpha} \delta_{Z^{i,N}} \Rightarrow \mu^\alpha.$$
	(b) Suppose in addition that we are in the setting of Example~\ref{eg:eg1}, namely $\xi_{ij}^N(t) \equiv \xi_{ij}^N(0)$ for all $t \ge 0$ and $i,j \in \Nbd$.
	Then for each $\alpha, \gamma \in \Kbd$ and $i \in \Nmb_\alpha$, as $N \to \infty$, $$\mu^{i,\gamma,N} \doteq \frac{1}{N_{i,\gamma}(0)} \sum_{j \in \Ngamma} \xi_{ij}^N(0) \delta_{Z^{j,N}} \one_{\{N_{i,\gamma}(0) > 0\}} \Rightarrow \mu^\gamma.$$
\end{corollary}


\section{Fluctuations and central limit theorems} \label{sec:fluctuations}

Next we will study the fluctuations of empirical measures about the law of large numbers limit.
For simplicity, we consider the {\em single-type} setting, i.e.\ $K=1$, and assume constant diffusion coefficients, i.e.\ $\sigma_{\alpha\gamma} \equiv I_d$, the $d$-dimensional identity matrix.
Consequently, we will write $\mu^N$, $\mu^{i,N}$, $\mu$, $N_i(t)$, $p_N(t)$, $b$ and $\bbar$ instead of $\mu^{\alpha,N}$, $\mu^{i,\gamma,N}$, $\mu^\alpha$, $N_{i,\gamma}(t)$, $p_{\alpha\gamma,N}(t)$, $b_{\alpha\gamma}$ and $\bbar_{\alpha\gamma}$.
Also, to simplify the notation, we will abbreviate $\xi_{ij}^N(t)$ and $\xi_{ij}^N$ as $\xi_{ij}(t)$ and $\xi_{ij}$ in the rest of the paper.


\subsection{Canonical processes} \label{sec:canonical_processes}

We first introduce the following canonical spaces and stochastic processes.
Let $\Omega_d \doteq \Cmc_d \times \Cmc_d$, $\Omega_e \doteq \Dmb([0,T]:\{0,1\})$ and $\Omega_N \doteq \Omega_d^N \times \Omega_e^{N \times N}$. 
Denote by $\nu \in \Pmc(\Omega_d)$ the common law of $(W^i, X^i)$ where $i \in \Nbd$ and $X^i$ is given by \eqref{eq:Xit}
(under the setting of this section).
Also denote by $\nu_{e,N} \in [\Pmc(\Omega_e)]^{N \times N}$ the law of the random adjacency matrix process $\{\xi_{ij}(t) : i,j \in \Nbd, t \in [0,T] \}$.
Define for $N \in \Nmb$ the probability measure $\Pmb^N$ on $\Omega_N$ as
\begin{equation*}
	\Pmb^{N} \doteq \Lmc \left((W^1, X^1), (W^2, X^2), \dotsc,(W^N, X^N), \{\xi_{ij} : i,j \in \Nbd \} \right) \equiv \nu^{\otimes N} \otimes \nu_{e,N}.
\end{equation*}
For $\omega = (\omega_1, \omega_2, \dotsc, \omega_N, \omegabar) \in \Omega_N$ with $\omegabar = (\omegabar_{ij})_{1 \le i,j \le N}$, let $V^i(\omega) \doteq \omega_i, i \in \Nbd$ and abusing notation,
\begin{equation*}
	V^i \doteq (W^i, X^i), \quad \xi_{ij}(\omega) \doteq \omegabar_{ij}, \quad i,j \in \Nbd.
\end{equation*}
Also define the canonical processes $V_* \doteq (W_*, X_*)$ on $\Omega_d$ as
\begin{equation*}
	V_*(\omega) \doteq (W_*(\omega), X_*(\omega)) \doteq (\omega_1, \omega_2), \quad \omega = (\omega_1, \omega_2) \in \Omega_d.
\end{equation*}


\subsection{Some integral operators} \label{sec:some_integral_operators}

We will need the following functions for stating our central limit theorem. Recall that $\mu = \nu_{(2)}$ denotes the law of $X^i$. Let $\mu_t$ be the marginal of $\mu$ at instant $t$, namely, $\mu_t \doteq \Lmc(X^i_t)$.
Define for $t \in [0,T]$, function $\bbar_t$ from $\Rmb^d \times \Rmb^d$ to $\Rmb^d$ as
\begin{equation} \label{eq:bbar t}
	\bbar_t(x,y) \doteq \bbar(x,y) - \int_{\Rmb^d} \bbar(x,z)\mu_t(dz), \quad (x,y) \in \Rmb^d \times \Rmb^d.
\end{equation}
Define function $h$ from $\Omega_d \times \Omega_d$ to $\Rmb$ ($\nu \otimes \nu$ a.s.) as
\begin{equation} \label{eq:h}
	h(\omega,\omega') \doteq \int_0^T \bbar_t (X_{*,t}(\omega),X_{*,t}(\omega')) \cdot dW_{*,t}(\omega), \quad (\omega, \omega') \in \Omega_d \times \Omega_d.
\end{equation}
Now consider the Hilbert space $L^2(\Omega_d, \nu)$.
Define integral operators $A$ on $L^2(\Omega_d, \nu)$ as
\begin{equation} \label{eq:A} 
	A f(\omega) \doteq \int_{\Omega_d} h(\omega',\omega) f(\omega') \, \nu(d\omega'), \quad f \in L^2(\Omega_d, \nu), \: \omega \in \Omega_d.
\end{equation}
Denote by $I$ the identity operator on $L^2(\Omega_d, \nu)$.
For $t \in [0,T]$, let
\begin{equation} \label{eq:lambdas}
	\lambda_t \doteq \int_{\Rd \times \Rd} \| \bbar_t(x,y) \|^2 \, \mu_t(dx) \, \mu_t(dy).
\end{equation} 
The following lemma is taken from \cite{BudhirajaWu2015} (see Lemma $3.1$ therein).

\begin{lemma} \label{lem:trace}
	(a) $\textnormal{Trace}(AA^*) = \int_{\Omega_d^2} h^2(\omega,\omega') \, \nu(d\omega) \, \nu(d\omega') = \int_0^T \lambda_t \, dt$.
	(b) \textnormal{Trace}$(A^n) = 0$ for all $n \ge 2$.
	(c) $I - A$ is invertible.
\end{lemma}


\subsection{Central limit theorem} \label{sec:CLT}

For the central limit theorem we need the following strengthened version of Condition \ref{cond:cond1}.

\begin{condition} \label{cond:cond2}
	$K=1$, $\sigma_{\alpha\gamma} \equiv I_d$, where $I_d$ is the $d$-dimensional identity matrix. The collection
	$\{ p_N(\cdot)  \}_{N \in \Nmb}$ is pre-compact in $\Cmb([0,T]:\Rmb)$ and $\liminf_{N \to \infty} \pbar_N > 0$.	
\end{condition}
We can now present the central limit theorem.
Let $L^2_c(\Cmc_d,\mu)$ be the space of all functions $\phi \in L^2(\Cmc_d,\mu)$ such that $\lan \phi, \mu \ran = 0$.
For $\phi \in L^2_c(\Cmc_d,\mu)$, let $\eta^N (\phi) \doteq \frac{1}{\sqrt{N}} \sum_{i=1}^N \phi (Z^{i,N})$ and $\phibd \doteq \phi(X_*) \in L^2(\Omega_d, \nu)$.

\begin{theorem} \label{thm:CLT}
	Suppose Conditions \ref{cond:coefficients} and \ref{cond:cond2} hold.
	Then $\{\eta^N (\phi) : \phi \in L^2_c(\Cmc_d,\mu)\}$ converges as $N \to \infty$ to a mean $0$ Gaussian field $\{\eta(\phi) : \phi \in L^2_c(\Cmc_d,\mu)\}$ in the sense of convergence of finite dimensional distributions, where for $\phi,\psi \in L^2_c(\Cmc_d,\mu)$,
	\begin{equation*}
		\Ebf [\eta(\phi) \eta(\psi)] = \langle (I-A)^{-1} \phibd, (I-A)^{-1} \psibd \rangle_{L^2(\Omega_d, \nu)}.
	\end{equation*}
\end{theorem}

Proof of the theorem is given in Section \ref{sec:pf_CLT}.

\section{Preliminary estimates} \label{sec:pre_results}

In this section we present several elementary results for a binomial distribution, which will be used for the proof of Theorems~\ref{thm:NpN_rate} and~\ref{thm:CLT}.
Proofs to these results are provided in Appendix~\ref{sec:pf_section_5} for completeness.

\begin{lemma} 
	\label{lem:prep_1}
	Let $X$ be a Binomial random variable with number of trials $n$ and probability of success $p$.
	Let $q \doteq 1 - p$.
	Then
	\begin{equation*}
		\Ebf \frac{1}{X+1} = \frac{1-q^{n+1}}{(n+1)p} \le \frac{1}{(n+1)p}.
	\end{equation*}
	Also for each $m = 2,3,\dotsc$,
	\begin{align*}
		& \Ebf \frac{1}{X+m} \le \frac{1-q^{n+1}}{(n+m)p} \le \frac{1}{(n+m)p}, \\
		& \Ebf \frac{1}{(X+1)^m} \le \frac{m^m}{(n+1)^m p^m}.
	\end{align*}
\end{lemma}

For the following lemma, let $\zeta_{ii} \equiv 1$ for $i \in \Nbd$ and $\{\zeta_{ij} = \zeta_{ji} : 1 \le i < j \le N\}$ be independent Bernoulli random variables with $\Pbd(\zeta_{ij}=1) = p_{\alpha\gamma,N} = p_{\gamma\alpha,N}$ for $i \in \Nalpha$, $j \in \Ngamma$, and $\alpha, \gamma \in \Kbd$.
Let $q_{\alpha\gamma,N} \doteq 1 - p_{\alpha\gamma,N}$ and $N_{i,\gamma} \doteq \sum_{j \in \Ngamma} \zeta_{ij}$ for $i \in \Nbd$ and $\alpha, \gamma \in \Kbd$. 


\begin{lemma} 
	\label{lem:prep_2}
	For  $\alpha,\gamma \in \Kbd$,
	\begin{align*}
		& \Ebf \left( \sum_{k \in \Nalpha} \frac{N_\gamma \zeta_{ki_\gamma}}{N_\alpha N_{k,\gamma}} \one_{\{N_{k,\gamma} > 0\}} - 1 \right)^2 \le \frac{4}{N_\alpha p_{\alpha\gamma,N}} + 2 e^{-N_\gamma p_{\alpha\gamma,N}}, \quad i_\gamma \in \Ngamma, \\
		& \Ebf \left( \sum_{k \in \Nalpha} \frac{\zeta_{ki_\alpha}}{N_{k,\alpha}} - 1 \right)^2 \le \frac{3}{N_\alpha p_{\alpha\alpha,N}}, \quad i_\alpha \in \Nalpha.
	\end{align*}
\end{lemma}

%

For the following lemma, let $Y = 1 + \sum_{i=2}^N \zeta_i$, where $\{\zeta_i, i=2,\dotsc,N\}$ are i.i.d.\ Bernoulli random variables with probability of success $p_N$.
We have the following tail bound on the random variable $Y$.
\begin{lemma} 
	\label{lem:prep_4}
	For $k > 0$, let $C_N(k) \doteq \sqrt{k (N-1) \log N }$. Then
	$\Pbd \left( |Y - N p_N| > C_N(k) + 1 \right) \le \frac{2}{N^{2k}}.$
\end{lemma}
\section{Proof of Theorem \ref{thm:NpN_rate}} \label{sec:pf_NpN_rate}
Fix $\alpha \in \Kbd$ and $\ialpha \in \Nalpha$.
	For fixed $t \in [0,T]$, we have from Cauchy--Schwarz and Doob's inequalities that
	\begin{align}
		\Ebf \left\|Z^{\ialpha,N} - X^{\ialpha}\right\|_{*,t}^2 
		& \le 2KT \sum_{\gamma=1}^K \int_0^t \Ebf \left\| b_{\alpha\gamma}(Z^{\ialpha,N}_s, \mu^{\ialpha,\gamma,N}_s) - b_{\alpha\gamma}(X^{\ialpha}_s, \mu^\gamma_s) \right\|^2 ds \notag \\
		& \qquad + 8K \sum_{\gamma=1}^K \int_0^t \Ebf \left\| \sigma_{\alpha\gamma}(Z^{\ialpha,N}_s, \mu^{\ialpha,\gamma,N}_s) - \sigma_{\alpha\gamma}(X^{\ialpha}_s, \mu^\gamma_s) \right\|^2 ds. \label{eq:NpN_1}
	\end{align}
	Fix $\gamma \in \Kbd$.
	Note that by adding and subtracting terms we have for $s \in [0,T]$,
	\begin{align}
		& \left\| b_{\alpha\gamma}(Z^{\ialpha,N}_s, \mu^{\ialpha,\gamma,N}_s) - b_{\alpha\gamma}(X^{\ialpha}_s, \mu^\gamma_s) \right\| \notag \\
		& \quad = \left\| \sum_{j \in \Ngamma} \frac{\xi_{\ialpha j}(s)}{N_{\ialpha,\gamma}(s)} \one_{\{N_{\ialpha,\gamma}(s) > 0\}} \bbar_{\alpha\gamma}(Z^{\ialpha,N}_s,Z^{j,N}_s) - b_{\alpha\gamma}(X^{\ialpha}_s,\mu^\gamma_s) \right\| \notag \\
		& \quad \le \sum_{n=1}^4 \Tmc^{\gamma,N,n}(s), \label{eq:NpN_2}
			\end{align}
where
$$		
\Tmc^{\gamma,N,1}(s) \doteq 		\left\| \sum_{j \in \Ngamma} \frac{\xi_{\ialpha j}(s)}{N_{\ialpha,\gamma}(s)} \one_{\{N_{\ialpha,\gamma}(s) > 0\}} \left( \bbar_{\alpha\gamma}(Z^{\ialpha,N}_s,\ZjNs) - \bbar_{\alpha\gamma}(X^{\ialpha}_s,\ZjNs) \right) \right\|,$$
$$
\Tmc^{\gamma,N,2}(s) \doteq \left\| \sum_{j \in \Ngamma} \frac{\xi_{\ialpha j}(s)}{N_{\ialpha,\gamma}(s)} \one_{\{N_{\ialpha,\gamma}(s) > 0\}} \left( \bbar_{\alpha\gamma}(X^{\ialpha}_s,\ZjNs) - \bbar_{\alpha\gamma}(X^{\ialpha}_s,X^j_s) \right) \right\|,$$
$$
\Tmc^{\gamma,N,3}(s) \doteq		\left\| \sum_{j \in \Ngamma} \frac{\xi_{\ialpha j}(s)}{N_{\ialpha,\gamma}(s)} \one_{\{N_{\ialpha,\gamma}(s) > 0\}} \left(\bbar_{\alpha\gamma}(X^{\ialpha}_s,X^j_s) - b_{\alpha\gamma}(X^{\ialpha}_s,\mu^\gamma_s) \right) \right\|,$$
and
$$
\Tmc^{\gamma,N,4}(s) \doteq \left\| \left( \sum_{j \in \Ngamma} \frac{\xi_{\ialpha j}(s)}{N_{\ialpha,\gamma}(s)} \one_{\{N_{\ialpha,\gamma}(s) > 0\}} - 1 \right) b_{\alpha\gamma}(X^{\ialpha}_s,\mu^\gamma_s) \right\|.$$
	
	Now we will analyze these terms one by one.
	For $\Tmc^{\gamma,N,1}$, first note that
	\begin{equation*}
		\sum_{j \in \Ngamma} \frac{\xi_{\ialpha j}(s)}{N_{\ialpha,\gamma}(s)} \one_{\{N_{\ialpha,\gamma}(s) > 0\}} = \one_{\{N_{\ialpha,\gamma}(s) > 0\}} \le 1.
	\end{equation*}
	So we have from Cauchy--Schwarz inequality and Lipschitz property of $\bbar$ that
	\begin{align}
		\Ebf [\Tmc^{\gamma,N,1}(s)]^2 & = \Ebf \left\| \sum_{j \in \Ngamma} \frac{\xi_{\ialpha j}(s)}{N_{\ialpha,\gamma}(s)} \one_{\{N_{\ialpha,\gamma}(s) > 0\}} \left( \bbar_{\alpha\gamma}(Z^{\ialpha,N}_s,\ZjNs) - \bbar_{\alpha\gamma}(X^{\ialpha}_s,\ZjNs) \right) \right\|^2 \notag \\
		& \le L^2 \Ebf \sum_{j \in \Ngamma} \frac{\xi_{\ialpha j}(s)}{N_{\ialpha,\gamma}(s)} \one_{\{N_{\ialpha,\gamma}(s) > 0\}} \left\| Z^{\ialpha,N}_s - X^{\ialpha}_s \right\|^2 \notag \\
		& \le L^2 \Ebf \|Z^{\ialpha,N}_s - X^{\ialpha}_s\|^2. \label{eq:TN1_bd}
	\end{align}
	
	For $\Tmc^{\gamma,N,2}$, when $\gamma \ne \alpha$, we have
	\begin{align}
		\Ebf [\Tmc^{\gamma,N,2}(s)]^2 & = \Ebf \left\| \sum_{j \in \Ngamma} \frac{\xi_{\ialpha j}(s)}{N_{\ialpha,\gamma}(s)} \one_{\{N_{\ialpha,\gamma}(s) > 0\}} \left( \bbar_{\alpha\gamma}(X^{\ialpha}_s,\ZjNs) - \bbar_{\alpha\gamma}(X^{\ialpha}_s,X^j_s) \right) \right\|^2 \notag \\
		& \le L^2 \Ebf \sum_{j \in \Ngamma} \frac{\xi_{\ialpha j}(s)}{N_{\ialpha,\gamma}(s)} \one_{\{N_{\ialpha,\gamma}(s) > 0\}} \left\|Z^{j,N}_s - X^j_s\right\|^2 \notag \\
		& = L^2 \Ebf \frac{N_\gamma \xi_{\ialpha\igamma}(s)}{N_{\ialpha,\gamma}(s)} \one_{\{N_{\ialpha,\gamma}(s) > 0\}} \left\|Z^{\igamma,N}_s - X^\igamma_s\right\|^2 \notag \\
		& = L^2 \Ebf \sum_{k \in \Nalpha} \frac{N_\gamma \xi_{k\igamma}(s)}{N_\alpha N_{k,\gamma}(s)} \one_{\{N_{k,\gamma}(s) > 0\}} \left\|Z^{\igamma,N}_s - X^\igamma_s\right\|^2, \label{eq:TN2_gamma}
	\end{align}
	where the last two equalities follow from the fact that when $\gamma \ne \alpha$
	\begin{align*}
		\Lmc(\xi_{i_\alpha j}(s), N_{i_\alpha,\gamma}(s), Z^{j,N}_s, X^j_s) & = \Lmc(\xi_{i_\alpha i_\gamma}(s), N_{i_\alpha,\gamma}(s), Z^{\igamma,N}_s, X^\igamma_s) \\
		& = \Lmc(\xi_{ki_\gamma}(s), N_{k,\gamma}(s), Z^{\igamma,N}_s, X^\igamma_s), \quad j \in \Ngamma, k \in \Nalpha.
	\end{align*}
	When $\gamma = \alpha$, we have
	\begin{align}
		\Ebf [\Tmc^{\alpha,N,2}(s)]^2 & = \Ebf \left\| \sum_{k \in \Nalpha} \frac{\xi_{\ialpha k}(s)}{N_{\ialpha,\alpha}(s)} \left( \bbar_{\alpha\alpha}(X^{\ialpha}_s,Z^{k,N}_s) - \bbar_{\alpha\alpha}(X^{\ialpha}_s,X^k_s) \right) \right\|^2 \notag \\
		& \le L^2 \Ebf \sum_{k \in \Nalpha} \frac{\xi_{\ialpha k}(s)}{N_{\ialpha,\alpha}(s)} \left\|Z^{k,N}_s - X^k_s\right\|^2 \notag \\
		& = L^2 \Ebf \sum_{k \in \Nalpha} \frac{\xi_{k\ialpha}(s)}{N_{k,\alpha}(s)} \left\|Z^{\ialpha,N}_s - X^\ialpha_s\right\|^2, \label{eq:TN2_alpha}
	\end{align}	
	where the last line follows from the fact that
	\begin{equation*}
		\Lmc(\xi_{\ialpha k}(s), N_{\ialpha,\alpha}(s), Z^{k,N}_s, X^k_s) = \Lmc(\xi_{k\ialpha}(s), N_{k,\alpha}(s), Z^{\ialpha,N}_s, X^{\ialpha}_s), \quad k \in \Nalpha.
	\end{equation*}
	Combining \eqref{eq:TN2_gamma} and \eqref{eq:TN2_alpha} shows that for each $\gamma \in \Kbd$,
	\begin{align}
		& \Ebf [\Tmc^{\gamma,N,2}(s)]^2 \notag \\
		& \le L^2 \Ebf \sum_{k \in \Nalpha} \frac{N_\gamma \xi_{k\igamma}(s)}{N_\alpha N_{k,\gamma}(s)} \one_{\{N_{k,\gamma}(s) > 0\}} \left\|Z^{\igamma,N}_s - X^\igamma_s\right\|^2 \notag \\
		& = L^2 \Ebf \left[ \left( \sum_{k \in \Nalpha} \frac{N_\gamma \xi_{k\igamma}(s)}{N_\alpha N_{k,\gamma}(s)} \one_{\{N_{k,\gamma}(s) > 0\}} - 1 \right) \left\|Z^{\igamma,N}_s - X^\igamma_s\right\|^2 \right] + L^2 \Ebf \left\|Z^{\igamma,N}_s - X^\igamma_s\right\|^2 \notag \\
		& \le L^2 \left[ \Ebf \left( \sum_{k \in \Nalpha} \frac{N_\gamma \xi_{k\igamma}(s)}{N_\alpha N_{k,\gamma}(s)} \one_{\{N_{k,\gamma}(s) > 0\}} - 1 \right)^2 \Ebf \left\|Z^{\igamma,N}_s - X^\igamma_s\right\|^4\right]^{1/2} + L^2 \Ebf \left\|Z^{\igamma,N}_s - X^\igamma_s\right\|^2. \label{eq:TN2_temp}
	\end{align}
	From Condition \ref{cond:coefficients} and an application of Cauchy--Schwarz and Doob's inequalities analogous to that for \eqref{eq:NpN_1} it follows that
	\begin{equation} 
		\label{eq:ZiXi_bd}
		\sup_{N \in \Nmb} \max_{i \in \Nbd} \Ebf \left\|Z^{i,N}_t - X^i_t\right\|_{*,T}^4 \doteq \kappa_1 < \infty.
	\end{equation}
	Applying this and Lemma \ref{lem:prep_2} to \eqref{eq:TN2_temp} gives us for all $\gamma \in \Kbd$
	\begin{equation} 
		\label{eq:TN2_bd}
		\Ebf [\Tmc^{\gamma,N,2}(s)]^2 \le \kappa_2 \left(\frac{1}{N_\alpha p_{\alpha\gamma,N}(s)} + e^{-N_\gamma p_{\alpha\gamma,N}(s)}\right)^{1/2} + L^2 \Ebf \left\|Z^{\igamma,N}_s - X^{\igamma}_s\right\|^2.
	\end{equation}
	
	For $\Tmc^{\gamma,N,3}$, since $\{X^i: i \in \Nbd\}$ are independent of $\{\xi_{ij}, N_{i,\gamma}: i,j \in \Nbd, \gamma \in \Kbd\}$, we have
	\begin{align*}
		\Ebf \left[\Tmc^{\gamma,N,3}(s)\right]^2 & = \Ebf \left\| \sum_{j \in \Ngamma} \frac{\xi_{\ialpha j}(s)}{N_{\ialpha,\gamma}(s)} \one_{\{N_{\ialpha,\gamma}(s) > 0\}} \left(\bbar_{\alpha\gamma}(X^{\ialpha}_s,X^j_s) - b_{\alpha\gamma}(X^{\ialpha}_s,\mu^\gamma_s) \right) \right\|^2  \\
		& = \sum_{j \in \Ngamma} \sum_{k \in \Ngamma} \left[ \Ebf \left( \frac{\xi_{\ialpha j}(s)}{N_{\ialpha,\gamma}(s)} \one_{\{N_{\ialpha,\gamma}(s) > 0\}} \frac{\xi_{\ialpha k}(s)}{N_{\ialpha,\gamma}(s)} \one_{\{N_{\ialpha,\gamma}(s) > 0\}} \right) \right. \\
		& \quad \cdot \left. \Ebf \left( \left(\bbar_{\alpha\gamma}(X^{\ialpha}_s,X^j_s) - b_{\alpha\gamma}(X^{\ialpha}_s,\mu^\gamma_s) \right) \left(\bbar_{\alpha\gamma}(X^{\ialpha}_s,X^k_s) - b_{\alpha\gamma}(X^{\ialpha}_s,\mu^\gamma_s) \right) \right) \right]  \\
		& = \sum_{j \in \Ngamma} \left[ \Ebf \left( \frac{\xi_{\ialpha j}(s)}{N_{\ialpha,\gamma}^2(s)} \one_{\{N_{\ialpha,\gamma}(s) > 0\}} \right) \Ebf \left( \bbar_{\alpha\gamma}(X^{\ialpha}_s,X^j_s) - b_{\alpha\gamma}(X^{\ialpha}_s,\mu^\gamma_s) \right)^2 \right].
	\end{align*}
	Thus using Lemma \ref{lem:prep_1},
	\begin{align}
		\Ebf \left[\Tmc^{\gamma,N,3}(s)\right]^2
		& \le 4L^2 \Ebf \sum_{j \in \Ngamma} \frac{\xi_{\ialpha j}(s)}{N_{\ialpha,\gamma}^2(s)} \one_{\{N_{\ialpha,\gamma}(s) > 0\}}  = 4L^2 \Ebf \frac{1}{N_{\ialpha,\gamma}(s)} \one_{\{N_{\ialpha,\gamma}(s) > 0\}} \notag \\
		& \le 4L^2 \Ebf \frac{2}{1 + N_{\ialpha,\gamma}(s)}  \le \frac{8L^2}{(N_\gamma+1) p_{\alpha\gamma,N}(s)}. \label{eq:TN3_bd}
	\end{align}

	For $\Tmc^{\gamma,N,4}$, we have
	\begin{align}
		\Ebf \left[\Tmc^{\gamma,N,4}(s)\right]^2 & = \Ebf \left\| \left( \sum_{j \in \Ngamma} \frac{\xi_{\ialpha j}(s)}{N_{\ialpha,\gamma}(s)} \one_{\{N_{\ialpha,\gamma}(s) > 0\}} - 1 \right) b_{\alpha\gamma}(X^{\ialpha}_s,\mu^\gamma_s) \right\|^2 \notag \\
		& \le L^2 \Pbd(N_{\ialpha,\gamma}(s) = 0)  \le L^2 (1-p_{\alpha\gamma,N}(s))^{N_\gamma}  \le L^2 e^{-N_\gamma p_{\alpha\gamma,N}(s)}. \label{eq:TN4_bd}
	\end{align}
	
	Combining \eqref{eq:NpN_2}, \eqref{eq:TN1_bd} and \eqref{eq:TN2_bd} -- \eqref{eq:TN4_bd} gives us
	\begin{align}
		& \Ebf \left\| b_{\alpha\gamma}(Z^{\ialpha,N}_s, \mu^{\ialpha,\gamma,N}_s) - b_{\alpha\gamma}(X^{\ialpha}_s, \mu^\gamma_s) \right\|^2 \notag \\
		& \quad \le \kappa_3 \max_{i \in \Nbd} \Ebf \left\|Z^{i,N} - X^i\right\|_{*,s}^2 + \kappa_3 \left(\frac{1}{N_\alpha p_{\alpha\gamma,N}(s)} + e^{-N_\gamma p_{\alpha\gamma,N}(s)}\right)^{1/2} + \frac{\kappa_3}{N_\gamma p_{\alpha\gamma,N}(s)} + \kappa_3 e^{-N_\gamma p_{\alpha\gamma,N}(s)}. \label{eq:NpN_3}
	\end{align}
	From exactly the same argument as above, it is clear that \eqref{eq:NpN_3} holds with $b_{\alpha\gamma}$ replaced by $\sigma_{\alpha\gamma}$, for all $\alpha, \gamma \in \Kbd$.
	So we have from \eqref{eq:NpN_1} that
	\begin{align*}
		\max_{i \in \Nbd} \Ebf \left\|Z^{i,N} - X^i\right\|_{*,t}^2
		& \le \kappa_4 \int_0^t \max_{i \in \Nbd} \Ebf \left\|Z^{i,N} - X^i\right\|_{*,s}^2 \, ds + \kappa_4 \left(\frac{1}{\Nbar \pbar_N} + e^{-\Nbar \pbar_N}\right)^{1/2} \\
		& \quad + \frac{\kappa_4}{\Nbar \pbar_N} + \kappa_{4} e^{-\Nbar \pbar_N}.
	\end{align*}
	The result now follows from Gronwall's lemma. \qed

	\begin{remark}
		\label{rem:remsec6}
		With a slight modification to the proof it can be shown that
		if one assumes diffusion coefficients $\{\sigma_{\alpha\gamma}\}_{\alpha,\gamma \in \Kbd}$ to be constants, then the estimate in \eqref{eq:NpN_better} holds.
		To see this,  first note that \eqref{eq:NpN_1}--\eqref{eq:TN2_alpha} are still valid with each second moment replaced by first moment.
		Combining \eqref{eq:TN2_gamma} and \eqref{eq:TN2_alpha}, using Cauchy-Schwarz inequality as in \eqref{eq:TN2_temp}, applying Lemma \ref{lem:prep_2} and \eqref{eq:ZiXi_bd}, one can argue \eqref{eq:TN2_bd} holds with second moment replaced by first moment.
		Also one can apply Cauchy-Schwarz inequality to get estimate of first moment from \eqref{eq:TN3_bd} and \eqref{eq:TN4_bd}.
		The desired result then follows once more from an application of Gronwall's lemma.	
	\end{remark}


\section{Proof of Theorem \ref{thm:CLT}} \label{sec:pf_CLT}
In this section we prove Theorem \ref{thm:CLT}. Conditions \ref{cond:coefficients} and \ref{cond:cond2}
will be in force throughout the section and thus will not be noted explicitly in the statement of various results.
For $N \in \Nmb$, let $\Omega_N, \Pmb^N, V_*, V^i, i \in \Nbd, \nu$ be as in Section \ref{sec:canonical_processes}.
For $t \in [0,T]$ and $i \in \Nbd$, define 
\begin{gather*}
	J^N(t) \doteq J^{N,1}(t) - \frac{1}{2} J^{N,2}(t), \quad \mutil^{i,N}_t \doteq \frac{1}{N_i(t)} \sum_{j=1}^N \xi_{ij}(t) \delta_{X^j_t},
\end{gather*}
where
\begin{equation} \label{eq:JN1}
	J^{N,1}(t) \doteq \sum_{i=1}^N \int_0^t \left( b(X^i_s,\mutil^{i,N}_s) - b(X^i_s,\mu_s) \right) \cdot d W^i_s
\end{equation}
and
\begin{equation} \label{eq:JN2}
	J^{N,2}(t) \doteq \sum_{i = 1}^N \int_0^t \left\| b(X^i_s,\mutil^{i,N}_s) - b(X^i_s,\mu_s) \right\|^2 \, ds.
\end{equation}
Let $\Fmctil^N_t \doteq \sigma\{V^i(s), \xi_{ij}(s),  0 \le s \le t, i, j \in \Nbd \}$.
Note that $\left\{\exp\left(J^N(t)\right)\right\}$ is an $\{\Fmctil^N_t\}$-martingale under $\Pmb^N$.
Define a new probability measure $\Qmb^N$ on $\Omega_N$ by 
\begin{equation*}
	\frac{d\Qmb^N}{d\Pmb^N} \doteq \exp\left(J^N(T)\right).
\end{equation*}

By Girsanov's Theorem, $(X^1, \dotsc, X^N, \{\xi_{ij} : i,j \in \Nbd \})$ has the same probability distribution under $\Qmb^N$ as $(Z^{1,N}, \dotsc, Z^{N,N}, \{\xi_{ij} : i,j \in \Nbd \})$ under $\Pbd$.
For $\phi \in L^2_c(\Cmc_d,\mu)$, let 
\begin{align}\label{eq:eq1231}\etatil^N (\phi) \doteq \frac{1}{\sqrt{N}} \sum_{i = 1}^N \phi (X^i).
\end{align}
Thus in order to prove the theorem it suffices to show that for any $\phi \in L^2_c(\Cmc_d,\mu)$,
\begin{equation*}
	\lim_{N \to \infty} \Ebf_{\Qmb^N} \exp \left( i \etatil^N(\phi) \right) = \exp \left( -\frac{1}{2} \left\| (I-A)^{-1} \phibd \right\|^2_{L^2(\Omega_d,\nu)} \right),
\end{equation*}
which is equivalent to showing
\begin{equation} \label{eq:cf_cvg}
	\lim_{N \to \infty} \Ebf_{\Pmb^N} \exp \left( i \etatil^N(\phi) + J^{N,1}(T) - \frac{1}{2} J^{N,2}(T) \right) = \exp \left( -\frac{1}{2} \left\| (I-A)^{-1} \phibd \right\|^2_{L^2(\Omega_d,\nu)} \right).
\end{equation}
For this we will need to study the asymptotics of $J^{N,1}$ and $J^{N,2}$ as $N \to \infty$.


\subsection{Asymptotics of symmetric statistics} \label{sec:asymp_symmetric_statistics}
The proof of \eqref{eq:cf_cvg} crucially relies on certain classical results from \cite{Dynkin1983} on limit laws of degenerate symmetric statistics.
In this section we briefly review these results.

Let $\Smb$ be a Polish space and let $\{Y_n\}_{n=1}^\infty$ be a sequence of i.i.d.\ $\Smb$-valued random variables having common probability law $\theta$.
For $k \in \Nmb$, let $L^2(\theta^{\otimes k})$ be the space of all real-valued square integrable functions on $(\Smb^k, \Bmc(\Smb)^{\otimes k}, \theta^{\otimes k})$.
Denote by $L^2_c(\theta^{\otimes k})$ the subspace of centered functions, namely $\phi \in L^2(\theta^{\otimes k})$ such that for all $1 \le j \le k$,
\begin{equation*}
	\int_\Smb \phi(x_1,\dotsc,x_{j-1},x,x_{j+1},\dotsc,x_k) \, \theta(dx) = 0, \quad \theta^{\otimes k-1} \text{ a.e. } (x_1,\dotsc,x_{j-1},x_{j+1},\dotsc,x_k).
\end{equation*}
Denote by $L^2_{sym}(\theta^{\otimes k})$ the subspace of symmetric functions, namely $\phi \in L^2(\theta^{\otimes k})$ such that for every permutation $\pi$ on $\{1,\dotsc,k\}$,
\begin{equation*}
	\phi(x_1,\dotsc,x_k) = \phi(x_{\pi(1)},\dotsc,x_{\pi(k)}), \quad \theta^{\otimes k} \text{ a.e. } (x_1,\dotsc,x_k).
\end{equation*}
Also, denote by $L^2_{c,sym}(\theta^{\otimes k})$ the subspace of centered symmetric functions in $L^2(\theta^{\otimes k})$, namely $L^2_{c,sym}(\theta^{\otimes k}) \doteq L^2_c(\theta^{\otimes k}) \bigcap L^2_{sym}(\theta^{\otimes k})$.
Given $\phi_k \in L^2_{sym}(\theta^{\otimes k})$ define the symmetric statistic $\Umc^n_k (\phi_k)$ as
\begin{equation*}
	\Umc^n_k (\phi_k) \doteq
	\begin{cases}
		\displaystyle \sum_{1 \le i_1 < i_2 < \dotsb < i_k \le n} \phi_k(Y_{i_1},\dotsc,Y_{i_k}) & \text{for } n \ge k \\
		0 & \text{for } n<k.
	\end{cases}
\end{equation*}
In order to describe the asymptotic distributions of such statistics consider a Gaussian field $\{I_1(h) : h \in L^2(\theta)\}$ such that
\begin{equation*}
	\Ebf \left( I_1(h) \right) = 0, \: \Ebf \left( I_1(h)I_1(g) \right) = \langle h,g \rangle_{L^2(\theta)}, \quad h,g \in L^2(\theta).
\end{equation*}
For $h \in L^2(\theta)$, define $\phi_k^h \in L^2_{sym}(\theta^{\otimes k})$ as
\begin{equation*}
	\phi_k^h(x_1,\dotsc,x_k) \doteq h(x_1) \dotsc h(x_k)
\end{equation*}
and set $\phi_0^h \doteq 1$.

The multiple Wiener integral (MWI) of $\phi_k^h$, denoted as $I_k(\phi_k^h)$, is defined through the following formula.
For $k \ge 1$,
\begin{equation*} 
	I_k(\phi_k^h) \doteq \sum_{j=0}^{\lfloor k/2 \rfloor} (-1)^j C_{k,j} ||h||^{2j}_{L^2(\theta)} (I_1(h))^{k-2j}, \text{ where } C_{k,j} \doteq \frac{k!}{(k-2j)! 2^j j!}, j=0,\dotsc,\lfloor k/2 \rfloor.
\end{equation*}
The following representation gives an equivalent way to characterize the MWI of $\phi_k^h $:
\begin{equation*}
	\sum_{k=0}^{\infty} \frac{t^k}{k!} I_k(\phi_k^h) = \exp \left( tI_1(h) - \frac{t^2}{2} ||h||^2_{L^2(\theta)} \right), \quad t \in \Rmb,
\end{equation*}
where we set $I_0(\phi_0^h) \doteq 1$. 
We extend the definition of $I_k$ to the linear span of $\{\phi_k^h, h \in L^2(\theta)\}$ by linearity. 
It can be checked that for all $f$ in this linear span,
\begin{equation} \label{eq:MWI_moments}
	\Ebf(I_k(f))^2 = k! \, ||f||^2_{L^2(\theta^{\otimes k})}.
\end{equation}
Using this identity and standard denseness arguments, the definition of $I_k(f)$ can be extended to all $f \in L^2_{sym}(\theta^{\otimes k})$ and the identity $\eqref{eq:MWI_moments}$ holds for all $f \in L^2_{sym}(\theta^{\otimes k})$. 
The following theorem is taken from \cite{Dynkin1983}.

\begin{theorem}[Dynkin-Mandelbaum \cite{Dynkin1983}] \label{thm:Dynkin}
	Let $\{ \phi_k \}_{k=1}^\infty$ be such that $\phi_k \in L^2_{c,sym}(\theta^{\otimes k})$ for each $k \ge 1$. 
	Then the following convergence holds as $n \to \infty$:
	\begin{equation*}
		\left( n^{-\frac{k}{2}} \Umc^n_k (\phi_k) \right)_{k \ge 1} \Rightarrow \left( \frac{1}{k!} I_k(\phi_k) \right)_{k \ge 1}
	\end{equation*}
	as a sequence of $\Rmb^\infty$-valued random variables.
\end{theorem}

\subsection{Completing the proof of Theorem \ref{thm:CLT}} \label{sec:complete_pf_CLT}

From Condition \ref{cond:cond2} for any subsequence of $\{p_N\}$ there is a further subsequence and a $p \in \Cmb([0,T]:\Rmb)$ such that, along the subsequence, $p_N$ converges to $p$.
In order to prove Theorem \ref{thm:CLT} it suffices to prove the statement in the theorem along every such subsequence. Thus, without loss of generality, we will assume that for some
$p \in \Cmb([0,T]:\Rmb)$,
\begin{equation}
	\label{eq:eq416}
	p_N \to p,\; \bar p \doteq \inf_{0\le s \le T} p(s) >0 \mbox{ and } p(s) \in [0,1] \mbox{ for all } s \in [0,T].
\end{equation}

Let $(\Omega^*,\Fmc^*, \Pbd^*)$ be some probability space on which we are given
\begin{enumerate}[$(a)$]
\item
	 a Gaussian random field $\{I_1(h): h \in L^2(\nu)\}$ as in Section \ref{sec:asymp_symmetric_statistics} with $\Smb = \Omega_d$ and $\theta = \nu$ (see Section \ref{sec:canonical_processes}).
\item
	a Gaussian random variable $Z$ with  zero mean and variance
	\begin{equation} \label{eq:sigmasq}
		\sigma^2 \doteq \int_0^T \frac{1-p(s)}{p(s)} \lambda_s \, ds,
	\end{equation}
	\end{enumerate}
	where $\lambda_t$ is defined in \eqref{eq:lambdas}, such that
	$Z$ is independent of $\{I_1(h): h \in L^2(\nu)\}$. Note that \eqref{eq:eq416} implies $\sigma^2 < \infty$.
	Define on this probability space multiple Wiener integrals $I_k(f)$ for $k \ge 1$ and $f \in L^2_{sym}(\nu^{\otimes k})$
	as in Section \ref{sec:asymp_symmetric_statistics}. Note that the collection
	$\{I_k(f): k\ge 1, f \in L^2_{sym}(\nu^{\otimes k})\}$ is independent of $Z$.
Recall functions $h$ and $\phibd$ (corresponding to each $\phi \in L^2_c(\Cmc_d,\mu)$) defined in Sections \ref{sec:some_integral_operators} and \ref{sec:CLT}. 
Define 
\begin{align}
	h^{sym}(\omega,\omega') & \doteq \half(h(\omega,\omega') + h(\omega',\omega)), \quad (\omega,\omega') \in \Omega_d^2 \label{eq:h_sym} \\
	m_t(x,y) & \doteq \int_{\Rmb^d} \bbar_t(z,x) \cdot \bbar_t(z,y) \, \mu_t(dz), \quad x,y \in \Rd, t \in [0,T] \label{eq:m_t} \\
	l(\omega,\omega') & \doteq \int_0^T m_s(X_{*,s}(\omega),X_{*,s}(\omega')) \, ds, \quad (\omega,\omega') \in \Omega_d^2. \label{eq:l}
\end{align}
In order to prove \eqref{eq:cf_cvg} (which will complete the proof of Theorem \ref{thm:CLT}), we will make use of the following key proposition, whose proof will be given in Section \ref{sec:pf_prop_key}. Recall $J^{N,1}$ and $J^{N,2}$
defined in \eqref{eq:JN1} and \eqref{eq:JN2} respectively. Also recall $\etatil^N(\phi)$ associated with
a $\phi \in  L^2_c(\Cmc_d,\mu)$ introduced in \eqref{eq:eq1231}.

\begin{proposition}
	\label{prop:key_joint_cvg}
	Let $\phi \in  L^2_c(\Cmc_d,\mu)$. Then,
	as $N \to \infty$,
	\begin{equation}
		\label{eq:joint_cvg}
		(\etatil^N(\phi), J^{N,1}(T), J^{N,2}(T)) \Rightarrow \left(I_1(\phibd), Z + I_2(h^{sym}), I_2(l) + \int_0^T \frac{1}{p(s)} \lambda_s \, ds \right).
	\end{equation}
\end{proposition}
We now complete the proof of \eqref{eq:cf_cvg} using the above proposition.
It follows from the  proposition that as $N \to \infty$,
\begin{equation}
	\left(\etatil^N(\phi), J^N(T)\right) \Rightarrow \left(I_1(\phibd), \frac{1}{2} I_2(f) + Z - \int_0^T \frac{1}{2p(s)} \lambda_s \, ds\right) \doteq (I_1(\phibd), J),\label{eq:eq1250}
\end{equation}
where $f$ is defined as
\begin{equation*}
	f(\omega,\omega') \doteq h(\omega,\omega') + h(\omega',\omega) - l(\omega,\omega'), \quad (\omega,\omega') \in \Omega_d^2.
\end{equation*}
Recalling the formula for $\textnormal{Trace}(AA^*)$ established in Lemma \ref{lem:trace}, the definition of
$\sigma^2$ in \eqref{eq:sigmasq}, and using independence between $Z$ and $\{I_k(\cdot)\}_{k \ge 1}$, we have
\begin{equation*}
	\Ebf_{\Pbd^*} \exp(J) = \Ebf_{\Pbd^*} \exp \left( \frac{1}{2} I_2(f) - \frac{1}{2} \textnormal{Trace}(AA^*) \right) \Ebf_{\Pbd^*} \exp \left( Z - \half \sigma^2 \right) = 1,
\end{equation*}
where the last equality follows from Lemma \ref{lem:Shiga_Tanaka} in Appendix \ref{sec:restating}.
Since $\EbfP \exp(J^N(T)) = 1$ for all $N$, we have from Scheffe's lemma that $\{\exp(J^N(T))\}$ is uniformly integrable and consequently so is $\{\exp(i\etatil^N(\phi) + J^N(T))\}$.
Hence, from \eqref{eq:eq1250}
  we have that as $N \to \infty$,
\begin{align*}
	& \quad \EbfP \exp \left( i\etatil^N(\phi) +J^N(T) \right) \\
	& \to \Ebf_{\Pbd^*} \exp \left(iI_1(\phibd) + J \right) \\
	& = \Ebf_{\Pbd^*} \exp \left( iI_1(\phibd) + \frac{1}{2} I_2(f) - \frac{1}{2} \textnormal{Trace}(AA^*) \right) \Ebf_{\Pbd^*} \exp \left( Z - \half \sigma^2 \right) \\
	& = \exp \left( -\frac{1}{2} \left\| (I-A)^{-1} \phibd \right\|^2_{L^2(\Omega_d,\nu)} \right),
\end{align*}
where the first equality uses the independence between $Z$ and $\{I_k(\cdot)\}_{k \ge 1}$ and the last equality again follows from Lemma \ref{lem:Shiga_Tanaka}.
Thus we have shown \eqref{eq:cf_cvg}, which completes the proof of Theorem \ref{thm:CLT}. \qed


\subsection{Proof of Proposition \ref{prop:key_joint_cvg}} \label{sec:pf_prop_key}

In this section we prove  Proposition \ref{prop:key_joint_cvg}.
We will first reduce $J^{N,1}(T)$ and $J^{N,2}(T)$ to forms that are more convenient to analyze.

\subsubsection{Reducing $J^{N,1}$ and $J^{N,2}$}
\label{sec:reducing}

The term $J^{N,1}(T)$ in \eqref{eq:JN1} can be written as
\begin{align*}
	J^{N,1}(T) & = \sum_{i=1}^N \int_0^T \left( b(X^i_s,\mutil^{i,N}_s) - b(X^i_s,\mu_s) \right) \cdot d W^i_s \\
	& = \sum_{i=1}^N \int_0^T \left( \frac{1}{N_i(s)} \sum_{j=1}^N \xi_{ij}(s) \bbar(X^i_s,X^j_s) - b(X^i_s,\mu_s) \right) \cdot d W^i_s \\
	& = \sum_{i=1}^N \sum_{j=1}^N \int_0^T \frac{\xi_{ij}(s)}{N_i(s)} \bbar_s(X^i_s,X^j_s) \cdot d W^i_s,
\end{align*}
where $\bbar_s$ is defined in \eqref{eq:bbar t}.
Let
\begin{equation} \label{eq:JN1til}
	\Jtil^{N,1}(T) \doteq \sum_{i=1}^N \sum_{j=1}^N \int_0^T \frac{\xi_{ij}(s)}{N p_N(s)} \bbar_s(X^i_s,X^j_s) \cdot d W^i_s.
\end{equation}

We will argue in Lemma \ref{lem:JN1_JN1_tilde} that the asymptotic behavior of $J^{N,1}$ is the same as that of $\Jtil^{N,1}(T)$, the proof of which relies on the following lemma.

\begin{lemma} \label{lem:approxNi}
	As $N \to \infty$,
	\begin{equation*}
		\sup_{s \in [0,T]} \EbfP \left[ \frac{N^2}{N_1^2(s)} - \frac{1}{p_N^2(s)} \right]^2 \to 0, \quad \sup_{s \in [0,T]} \EbfP \left[ \frac{N}{N_1(s)} - \frac{1}{p_N(s)} \right]^2 \to 0.
	\end{equation*}
\end{lemma}

\begin{proof}
	It suffices to prove the first convergence, since the second one follows from the inequality
	\begin{equation*}
		\left|\frac{N}{N_1(s)} - \frac{1}{p_N(s)}\right| = \frac{\left|\frac{N^2}{N_1^2(s)} - \frac{1}{p_N^2(s)}\right|}{\left|\frac{N}{N_1(s)} + \frac{1}{p_N(s)}\right|} \le \left|\frac{N^2}{N_1^2(s)} - \frac{1}{p_N^2(s)}\right|.
	\end{equation*}
	Recall $C_N(\cdot)$ from Lemma \ref{lem:prep_4}.
	For $s \in [0,T]$, consider the event $G_N(s) \doteq \{ \omega \in \Omega_N : |N_1(s) - Np_N(s)| > C_N(3) + 1 \}$.
	It follows from Lemma \ref{lem:prep_4} that $\Pmb^N(G_N(s)) \le \frac{2}{N^6}$.
	Write
	\begin{equation} \label{eq:approxNi}
		\left[ \frac{N^2}{N_1^2(s)} - \frac{1}{p_N^2(s)} \right]^2 = \frac{(N_1^2(s) - N^2p_N^2(s))^2}{N_1^4(s) p_N^4(s)} \one_{G_N(s)} + \frac{(N_1^2(s) - N^2p_N^2(s))^2}{N_1^4(s) p_N^4(s)} \one_{G_N^c(s)}.
	\end{equation}
	Noting that $| N_1^2(s) - N^2 p_N^2(s)| \le N^2$ and $N_1(s) \ge 1$, we have as $N \to \infty$,
	\begin{equation} \label{eq:approxNi1}
		\EbfP \left[\frac{(N_1^2(s) - N^2p_N^2(s))^2}{N_1^4(s) p_N^4(s)} \one_{G_N(s)}\right] \le \frac{N^4}{p_N^4(s)} \Pmb^N(G_N(s)) \le \frac{2}{N^2 \pbar_N^4} \to 0,
	\end{equation}
	where the convergence follows from Condition \ref{cond:cond2}.
	Now consider the second term on the right side of \eqref{eq:approxNi}.
	Condition \ref{cond:cond2} implies that $N \pbar_N - C_N(3) - 1 > 0$ for large enough $N$.
	Hence
	\begin{equation} \label{eq:approxNi2}
		\begin{aligned} 
			\EbfP \left[\frac{(N_1^2(s) - N^2p_N^2(s))^2}{N_1^4(s) p_N^4(s)} \one_{G_N^c(s)}\right] & \le \EbfP \left[\frac{(C_N(3)+1)^2 (N_1(s)+Np_N(s))^2}{(Np_N(s) - C_N(3) - 1)^4p_N^4(s)} \one_{G_N^c(s)}\right] \\
			& \le \frac{4N^2(\sqrt{3 (N-1) \log N} + 1)^2}{(N \pbar_N - \sqrt{3 (N-1) \log N} - 1)^4 \pbar_N^4} \to 0
		\end{aligned}		
	\end{equation}
	as $N \to \infty$.
	The result follows by combining \eqref{eq:approxNi1} and \eqref{eq:approxNi2}.
\end{proof}

The following lemma says that to study the asymptotics of $J^{N,1}(T)$, it suffices to study the asymptotic behavior of $\Jtil^{N,1}(T)$.

\begin{lemma} \label{lem:JN1_JN1_tilde}
	\begin{equation*}
		\lim_{N \to \infty} \EbfP \left|J^{N,1}(T) - \Jtil^{N,1}(T)\right|^2 = 0.
	\end{equation*}
\end{lemma}

\begin{proof}
	First note that as $N \to \infty$
	\begin{align*}
		& \EbfP \left| \sum_{1 \le i < j \le N} \int_0^T \frac{\xi_{ij}(s)}{N_i(s)} \bbar_s(X^i_s,X^j_s) \cdot d W^i_s - \sum_{1 \le i < j \le N} \int_0^T \frac{\xi_{ij}(s)}{N p_N(s)} \bbar_s(X^i_s,X^j_s) \cdot d W^i_s \right|^2 \\
		& \quad = \EbfP \left( \sum_{1 \le i < j \le N} \int_0^T \left(\frac{1}{N_i(s)} - \frac{1}{Np_N(s)}\right) \xi_{ij}(s) \bbar_s(X^i_s,X^j_s) \cdot d W^i_s \right)^2 \\
		& \quad = \sum_{1 \le i < j \le N} \int_0^T \left( \EbfP \left[\left(\frac{1}{N_i(s)} - \frac{1}{Np_N(s)}\right) \xi_{ij}(s)\right]^2 \EbfP \bbar_s^2(X^i_s,X^j_s) \right) \, ds \\
		& \quad \le \kappa N^2 \int_0^T \EbfP \left(\frac{1}{N_1(s)} - \frac{1}{Np_N(s)}\right)^2 \, ds \to 0,
	\end{align*}
	where the convergence follows from Lemma \ref{lem:approxNi}.
	Similarly one can show that
	\begin{gather*}
		\lim_{N \to \infty} \EbfP \left| \sum_{1 \le j < i \le N} \int_0^T \frac{\xi_{ij}(s)}{N_i(s)} \bbar_s(X^i_s,X^j_s) \cdot d W^i_s - \sum_{1 \le j < i \le N} \int_0^T \frac{\xi_{ij}(s)}{N p_N(s)} \bbar_s(X^i_s,X^j_s) \cdot d W^i_s \right|^2 = 0, \\
		\lim_{N \to \infty} \EbfP \left| \sum_{i=1}^N \int_0^T \frac{1}{N_i(s)} \bbar_s(X^i_s,X^i_s) \cdot d W^i_s - \sum_{i=1}^N \int_0^T \frac{1}{N p_N(s)} \bbar_s(X^i_s,X^i_s) \cdot d W^i_s \right|^2 = 0.
	\end{gather*} 
	Combining above results completes the proof.
\end{proof}

%
%
%

Next we study the asymptotics of $J^{N,2}(T)$. For that we will need the following lemma.

\begin{lemma} \label{lem:Ttil}
	Suppose $\gamma, \vartheta, \rho$ are bounded measurable real maps on $[0,T]\times (\Rmb^d)^3$,
	$[0,T]\times (\Rmb^d)^2$ and $[0,T]\times \Rmb^d$ respectively, such that for all $s\in [0,T]$ 
 $\gamma_s \equiv \gamma(s, \cdot) \in L^2_c(\mu_s^{\otimes 3})$, $\vartheta_s \equiv \vartheta(s, \cdot)\in L^2_c(\mu_s^{\otimes 2})$ and $\rho_s \equiv \rho(s, \cdot)\in L^2_c(\mu_s)$.
	Then as $N \to \infty$
	\begin{align}
		& \EbfP \left| \sum_{i\neq j, j\neq k, i\neq k} \int_0^T \frac{\xi_{ij}(s)\xi_{ik}(s)}{N_i^2(s)} \gamma_s(X^i_s,X^j_s,X^k_s) \, ds \right| \to 0, \label{eq:Ttilijk}\\
		& \EbfP \left| \sum_{1 \le i \ne k \le N} \int_0^T \frac{\xi_{ik}(s)}{N_i^2(s)} \vartheta_s(X^i_s,X^k_s) \, ds \right| \to 0, \label{eq:Ttilik}\\
		& \EbfP \left| \sum_{1 \le i \ne k \le N} \int_0^T \frac{\xi_{ik}(s)}{N_i^2(s)} \rho_s(X^k_s) \, ds \right| \to 0, \label{eq:Ttilk} \\
		& \EbfP \left| \sum_{1 \le i \ne k \le N} \int_0^T \frac{\xi_{ik}(s)}{N_i^2(s)} \rho_s(X^i_s) \, ds \right| \to 0 \label{eq:Ttili}.
	\end{align}
\end{lemma}

\begin{proof}
	To prove \eqref{eq:Ttilijk}, it is enough to prove the convergence with the summation taken over the ordered sum $i < j < k$.
	Now note that
	\begin{align*}
		& \EbfP \left( \sum_{1 \le i < j < k \le N} \int_0^T \frac{\xi_{ij}(s)\xi_{ik}(s)}{N_i^2(s)} \gamma_s(X^i_s,X^j_s,X^k_s) \, ds \right)^2 \\
		& \quad = \EbfP \sum_{1 \le i < j < k \le N} \left( \int_0^T \frac{\xi_{ij}(s)\xi_{ik}(s)}{N_i^2(s)} \gamma_s(X^i_s,X^j_s,X^k_s) \, ds \right)^2 \\
		& \quad \le \kappa_1 N^3 \int_0^T \EbfP \frac{1}{N_1^4(s)} \, ds \le \kappa_2 N^3 \int_0^T \frac{1}{N^4 p_N^4(s)} \, ds \le \frac{\kappa_2 T}{N \pbar_N^4} \to 0,
	\end{align*}
	where the second inequality follows from Lemma~\ref{lem:prep_1}.
	Thus \eqref{eq:Ttilijk} holds.
	Proofs for \eqref{eq:Ttilik}, \eqref{eq:Ttilk} and \eqref{eq:Ttili} are similar and hence omitted.
\end{proof}
Recall the definition of $m_t$ from \eqref{eq:m_t}. Define
\begin{equation}\label{eq:eq114}
	\Jtil^{N,2}(T) \doteq \frac{N-2}{N^2}\sum_{1 \le j\neq k \le N} \int_0^T m_s(X^j_s,X^k_s) \, ds + \int_0^T \frac{\lambda_s}{p(s)}  \, ds.
\end{equation}
The following lemma shows that $\Jtil^{N,2}(T)$ is asymptotically the same as $J^{N,2}(T)$.
\begin{lemma}
	\label{lem:tiljn2}
	As $N\to \infty$, $J^{N,2}(T)-\Jtil^{N,2}(T)$ converges to $0$ in probability.
\end{lemma}
\begin{proof}
We split $J^{N,2}(T)$ as follows:
\begin{align*}
	J^{N,2}(T) & 
	= \sum_{i,j,k=1}^N \int_0^T \frac{\xi_{ij}(s) \xi_{ik}(s)}{N_i^2(s)} \bbar_s(X^i_s,X^j_s) \cdot \bbar_s(X^i_s,X^k_s) \, ds \\
	& = \sum_{n=1}^5 \sum_{(i,j,k) \in \Smc_n} \int_0^T \frac{\xi_{ij(s)} \xi_{ik}(s)}{N_i^2(s)} \bbar_s(X^i_s,X^j_s) \cdot \bbar_s(X^i_s,X^k_s) \, ds \doteq \sum_{n=1}^5 \Tmctil^{N,n},
\end{align*}
where $\Smc_1$, $\Smc_2$, $\Smc_3$, $\Smc_4$ and $\Smc_5$ are collections of $(i,j,k) \in \Nbd^3$ such that $\{i=j=k\}$, $\{i=j \ne k\}$, $\{i=k \ne j\}$, $\{j=k \ne i\}$ and $\{i,j,k \text{ distinct}\}$, respectively.
For $\Tmctil^{N,1}$, we have
\begin{align*}
	\EbfP | \Tmctil^{N,1} | & = \EbfP \sum_{i=1}^N \int_0^T \frac{1}{N_i^2(s)} \| \bbar_s(X^i_s,X^i_s) \|^2 \, ds \le \kappa_1 N \int_0^T \EbfP \frac{1}{N_1^2(s)} \, ds \\
	& \le \kappa_2 N \int_0^T \frac{1}{N^2 p_N^2(s)} \, ds \le \frac{\kappa_2 N T}{N^2 \pbar_N^2} \to 0
\end{align*}
as $N \to \infty$, where the second inequality follows from Lemma \ref{lem:prep_1}.	
To study the asymptotics of $\Tmctil^{N,2}$, $\Tmctil^{N,3}$, $\Tmctil^{N,4}$ and $\Tmctil^{N,5}$, we will use Lemma \ref{lem:Ttil}.

For $\Tmctil^{N,2}$, note that
\begin{align*}
	\Tmctil^{N,2} & = \sum_{1 \le i \ne k \le N} \int_0^T \frac{\xi_{ik}(s)}{N_i^2(s)} \bbar_s(X^i_s,X^i_s) \cdot \bbar_s(X^i_s,X^k_s) \, ds \\
	& = \sum_{1 \le i \ne k \le N} \int_0^T \frac{\xi_{ik}(s)}{N_i^2(s)} \left( \bbar_s(X^i_s,X^i_s) \cdot \bbar_s(X^i_s,X^k_s) - \intRd \bbar_s(y,y) \cdot \bbar_s(y,X^k_s) \, \mu_s(dy) \right) \, ds \\
	& \quad + \sum_{1 \le i \ne k \le N} \int_0^T \frac{\xi_{ik}(s)}{N_i^2(s)} \intRd \bbar_s(y,y) \cdot \bbar_s(y,X^k_s) \, \mu_s(dy) \, ds.
\end{align*}
It then follows from Lemma \ref{lem:Ttil} (see \eqref{eq:Ttilik} and \eqref{eq:Ttilk}) that $\EbfP |\Tmctil^{N,2}| \to 0$ as $N \to \infty$.
Similarly $\EbfP |\Tmctil^{N,3}| \to 0$ as $N \to \infty$.
	
Consider now the fourth term $\Tmctil^{N,4}$. 
Recalling $\lambda_t$ defined in \eqref{eq:lambdas}, we write
\begin{align*}
	\Tmctil^{N,4} & = \sum_{1 \le i \ne j \le N} \int_0^T \frac{\xi_{ij}(s)}{N_i^2(s)} \| \bbar_s(X^i_s,X^j_s) \|^2 \, ds \\
	& = \sum_{1 \le i \ne j \le N} \int_0^T \frac{\xi_{ij}(s)}{N_i^2(s)} \bigg( \| \bbar_s(X^i_s,X^j_s) \|^2 - \lan \|\bbar_s(X^i_s,\cdot)\|^2, \mu_s \ran - \lan \| \bbar_s(\cdot,X^j_s) \|^2,\mu_s \ran + \lambda_s \bigg) ds \\
	& \qquad + \sum_{1 \le i \ne j \le N} \int_0^T \frac{\xi_{ij}(s)}{N_i^2(s)} \left( \lan \|\bbar_s(X^i_s,\cdot)\|^2, \mu_s \ran - \lambda_s \right) ds \\
	& \qquad + \sum_{1 \le i \ne j \le N} \int_0^T \frac{\xi_{ij}(s)}{N_i^2(s)} \left( \lan \| \bbar_s(\cdot,X^j_s) \|^2,\mu_s \ran - \lambda_s \right) ds + \sum_{1 \le i \ne j \le N} \int_0^T \frac{\xi_{ij}(s)}{N_i^2(s)} \lambda_s \, ds \\
	& \doteq \sum_{n=1}^4 \Tmctil^{N,4}_n.
\end{align*}
It follows from Lemma \ref{lem:Ttil} (see \eqref{eq:Ttilik}, \eqref{eq:Ttilk} and \eqref{eq:Ttili}) that as $N \to \infty$,
\begin{equation*}
	\EbfP |\Tmctil^{N,4}_n| \to 0, \quad n=1,2,3.
\end{equation*}
We now show that
\begin{equation} \label{eq:TN44}
	\EbfP \left| \Tmctil^{N,4}_4 - \int_0^T \frac{1}{p(s)} \lambda_s \, ds \right| \to 0
\end{equation}
as $N \to \infty$.
For this, first write
\begin{equation*}
	\Tmctil^{N,4}_4 = \sum_{i=1}^N \int_0^T \frac{N_i(s)-1}{N_i^2(s)} \lambda_s \, ds = \sum_{i=1}^N \int_0^T \frac{1}{N_i(s)} \lambda_s \, ds - \sum_{i=1}^N \int_0^T \frac{1}{N_i^2(s)} \lambda_s \, ds.
\end{equation*}
By Lemma~\ref{lem:approxNi}
\begin{equation*}
	\lim_{N \to \infty} \EbfP \sum_{i=1}^N \int_0^T \frac{1}{N_i^2(s)} \lambda_s \, ds = \lim_{N \to \infty} \int_0^T \EbfP \frac{N}{N_1^2(s)} \lambda_s \, ds = 0.
\end{equation*}
Also as $N \to \infty$,
\begin{align*}
	\EbfP \left| \sum_{i=1}^N \int_0^T \frac{1}{N_i(s)} \lambda_s \, ds - \int_0^T \frac{1}{p(s)} \lambda_s \, ds \right|^2 & \le \kappa_3 N \EbfP \sum_{i=1}^N \int_0^T \left( \frac{1}{N_i(s)} - \frac{1}{Np(s)} \right)^2 \, ds \\
	& \le \kappa_3 N^2 \int_0^T \EbfP \left( \frac{1}{N_1(s)} - \frac{1}{Np(s)} \right)^2 \, ds \to 0,
\end{align*}
where the convergence follows from Lemma \ref{lem:approxNi} and \eqref{eq:eq416}.
This proves  \eqref{eq:TN44} and hence as $N \to \infty$,
\begin{equation*}
\Tmctil^{N,4} \to \int_0^T \frac{1}{p(s)} \lambda_s \, ds \mbox{ in probability}.
\end{equation*}

Finally consider the last term $\Tmctil^{N,5}$.
Recalling $m_t$ defined in \eqref{eq:m_t}, we have
\begin{align*}
	\Tmctil^{N,5} & = \sum_{(i,j,k) \in \Smc_5} \int_0^T \frac{\xi_{ij}(s) \xi_{ik}(s)}{N_i^2(s)} \bbar_s(X^i_s,X^j_s) \cdot \bbar_s(X^i_s,X^k_s) \, ds \\
	& = \sum_{(i,j,k) \in \Smc_5} \int_0^T \frac{\xi_{ij}(s) \xi_{ik}(s)}{N_i^2(s)} \left( \bbar_s(X^i_s,X^j_s) \cdot \bbar_s(X^i_s,X^k_s) - m_s(X^j_s,X^k_s) \right) \, ds \\
	& \qquad + \sum_{(i,j,k) \in \Smc_5} \int_0^T \frac{\xi_{ij}(s) \xi_{ik}(s)}{N_i^2(s)} m_s(X^j_s,X^k_s) \, ds \\
	& \doteq \Tmctil^{N,5}_1 + \Tmctil^{N,5}_2.
\end{align*}
It follows from Lemma \ref{lem:Ttil} (see \eqref{eq:Ttilijk}) that $\EbfP |\Tmctil^{N,5}_1| \to 0$ as $N \to \infty$.
Let
\begin{equation*}
	\Tmchat^{N,5} \doteq \sum_{(i,j,k) \in \Smc_5} \int_0^T \frac{\xi_{ij}(s) \xi_{ik}(s)}{N^2 p_N^2(s)} m_s(X^j_s,X^k_s) \, ds.
\end{equation*}
We now show that as $N \to \infty$,
\begin{equation} \label{eq:TN5}
	\EbfP |\Tmctil^{N,5}_2 - \Tmchat^{N,5}|^2 \to 0.
\end{equation}
To see this, as before, it suffices to consider the summation over ordered indices $i < j < k$.
Note that
\begin{align*}
	& \EbfP \left( \sum_{1 \le i < j < k \le N} \int_0^T \xi_{ij}(s) \xi_{ik}(s) \left( \frac{1}{N_i^2(s)} - \frac{1}{N^2 p_N^2(s)} \right) m_s(X^j_s,X^k_s) \, ds   \right)^2 \\
	& \quad \le N \EbfP \sum_{i=1}^N \left( \sum_{1 \le j < k \le N} \int_0^T \xi_{ij}(s) \xi_{ik}(s) \left( \frac{1}{N_i^2(s)} - \frac{1}{N^2 p_N^2(s)} \right) m_s(X^j_s,X^k_s) \, ds   \right)^2 \\
	& \quad = N \EbfP \sum_{i=1}^N \sum_{1 \le j < k \le N} \left( \int_0^T \xi_{ij}(s) \xi_{ik}(s) \left( \frac{1}{N_i^2(s)} - \frac{1}{N^2 p_N^2(s)} \right) m_s(X^j_s,X^k_s) \, ds   \right)^2 \\
	& \quad \le \kappa_4 N^4 \EbfP \int_0^T \left( \frac{1}{N_1^2(s)} - \frac{1}{N^2 p_N^2(s)} \right)^2 ds \to 0,
\end{align*}
where the convergence follows from Lemma \ref{lem:approxNi}.
So \eqref{eq:TN5} holds.
Next split $\Tmchat^{N,5}$ as
\begin{align}
	\Tmchat^{N,5} & = \sum_{(i,j,k) \in \Smc_5, j<k} \int_0^T \frac{\xi_{ij}(s) \xi_{ik}(s) - p_N^2(s)}{N^2 p_N^2(s)} m_s(X^j_s,X^k_s) \, ds \notag \\
	& \qquad + \sum_{(i,j,k) \in \Smc_5, j>k} \int_0^T \frac{\xi_{ij}(s) \xi_{ik}(s) - p_N^2(s)}{N^2 p_N^2(s)} m_s(X^j_s,X^k_s) \, ds \notag \\
	& \qquad + \sum_{(i,j,k) \in \Smc_5} \frac{1}{N^2} \int_0^T m_s(X^j_s,X^k_s) \, ds \notag \\
	& \doteq \Tmchat^{N,5}_1 + \Tmchat^{N,5}_2 + \Tmchat^{N,5}_3. \label{eq:Tmchat_N_5_3}
\end{align}
It follows from Condition \ref{cond:cond2} that as $N \to \infty$
\begin{align*}
	\EbfP |\Tmchat^{N,5}_1|^2 & = \sum_{(i,j,k) \in \Smc_5, j<k} \EbfP \left( \int_0^T \frac{\xi_{ij}(s) \xi_{ik}(s) - p_N^2(s)}{N^2 p_N^2(s)} m_s(X^j_s,X^k_s) \, ds \right)^2 \\
	& \le \kappa_5 \sum_{(i,j,k) \in \Smc_5, j<k} \int_0^T \EbfP \left( \frac{\xi_{ij}(s) \xi_{ik}(s) - p_N^2(s)}{N^2 p_N^2(s)} \right)^2 \, ds \\
	& \le \frac{\kappa_6}{N \pbar_N^4} \to 0.
\end{align*}
Similarly $\EbfP |\Tmchat^{N,5}_2|^2 \to 0$ as $N \to \infty$.
Combining the above convergence results, we have
$$J^{N,2}(T) - \Tmchat^{N,5}_3 - \int_0^T \frac{\lambda_s}{p(s)}  \, ds \to 0$$
in probability as $N\to \infty$.  The result now follows on observing that
\begin{equation*}
	\Jtil^{N,2}(T) = \Tmchat^{N,5}_3 + \int_0^T \frac{\lambda_s}{p(s)}  \, ds. \qedhere
\end{equation*}
\end{proof}
From Lemma \ref{lem:JN1_JN1_tilde} and \ref{lem:tiljn2} we have that
\begin{equation}
	\label{eq:joint_new}
	\left(J^{N,1}(T), J^{N,2}(T)\right) = \left( \Jtil^{N,1}(T) + \Rmc^{N,1}, \Jtil^{N,2}(T) + \Rmc^{N,2} \right),
\end{equation}
where $\Rmc^{N,1}, \Rmc^{N,2} \to 0$ in probability as $N \to \infty$.

\subsubsection{Asymptotics of $(\etatil^N(\phi), \Jtil^{N,1}(T), \Jtil^{N,2}(T))$}
\label{sec:asymptotics_joint}

Define for $\zeta \in \Dmb([0,T]:\{0,1\})$ and $\nu\otimes \nu$ a.e.\ $(\omega,\omega') \in \Omega_d^2$,
\begin{align*}
	u_N(\zeta,\omega,\omega') & \doteq \int_0^T \frac{\zeta(s)-p_N(s)}{p_N(s)} \bbar_s(X_{*,s}(\omega),X_{*,s}(\omega')) \cdot d W_{*,s}(\omega) \\
	& \quad + \int_0^T \frac{\zeta(s)-p_N(s)}{p_N(s)} \bbar_s(X_{*,s}(\omega'),X_{*,s}(\omega)) \cdot d W_{*,s}(\omega').
\end{align*}
Let, with $\{V^i\}$ and $\{\xi_{ij}\}$ as in Section \ref{sec:canonical_processes},
\begin{equation}
	\label{eq:U_N}
	U_N \doteq \frac{1}{N} \sum_{1 \le i < j \le N} u_N(\xi_{ij},V^i,V^j).
\end{equation}
The following result is  key for studying the asymptotics of $(\etatil^N(\phi), \Jtil^{N,1}(T), \Jtil^{N,2}(T))$.
The special case where $\xi_{ij}(s) = \xi_{ij}(0)$ for all $s \in [0,T]$, $i,j \in \Nbd$ and $\phi_k=0$ for all $k\neq 2$ was considered by Janson in the study of incomplete $U$-statistics~\cite{Janson1984}.

Recall $Z$ and $\{I_k(\cdot)\}$ introduced in Section \ref{sec:complete_pf_CLT}.
\begin{lemma} \label{lem:key}
	Let $\{ \phi_k \}_{k=1}^\infty$ be such that $\phi_k \in L^2_{c,sym}(\nu^{\otimes k})$ for each $k \ge 1$. 
	Then the following convergence holds as $N \to \infty$:
	\begin{equation*}
		\left( U_N, \left( N^{-\frac{k}{2}} \Umc^N_k (\phi_k) \right)_{k \ge 1} \right) \Rightarrow \left( Z, \left( \frac{1}{k!} I_k(\phi_k) \right)_{k \ge 1} \right)
	\end{equation*}
	as a sequence of $\Rmb^\infty$-valued random variables, where $Z$ and $\{I_k(\cdot)\}_{k \ge 1}$ are as introduced in Section \ref{sec:complete_pf_CLT}, and $\Umc^N_k(\cdot)$ is as defined in Section \ref{sec:asymp_symmetric_statistics} with $Y_i$ replaced by $V_i$.
\end{lemma}



\begin{proof}
%
%
	The proof uses similar conditioning arguments as in Janson \cite{Janson1984} (see Lemma $2$ and Theorem $1$ therein).
	Fix $K \in \Nmb$, $\phi_k \in L^2_{c,sym}(\nu^{\otimes k})$ and $t, s_k \in \Rmb$ for $k=1,\dotsc,K$.
	Denote by $\EbfPV$ the conditional expectation under $\Pmb^N$ given $(V^i)_{i=1}^N$.	
	Since the collection $\{u_N(\xi_{ij},V^i,V^j)\}_{i<j}$ is conditionally independent given $(V^i)_{i=1}^N$, and $\EbfPV [u_N(\xi_{ij},V^i,V^j)] = 0$ for each $i<j$, we have
	\begin{equation*}
		\sigma_N^2 \doteq \EbfPV [U_N^2] = \frac{1}{N^2} \sum_{1 \le i < j \le N} \EbfPV [u_N^2(\xi_{ij},V^i,V^j)].
	\end{equation*}
	Recall $\sigma^2$ defined in \eqref{eq:sigmasq}.
	It follows from \eqref{eq:eq416} that as $N \to \infty$,
	\begin{equation*}
		\EbfP [\sigma_N^2] = \frac{1}{N^2} \sum_{1 \le i < j \le N} \EbfP [u_N^2(\xi_{ij},V^i,V^j)] = \frac{2}{N^2} \sum_{1 \le i < j \le N} \int_0^T \frac{1-p_N(s)}{p_N(s)} \lambda_s \, ds \to \sigma^2
	\end{equation*}
	and
	\begin{align*}
		\EbfP \left(\sigma_N^2 - \EbfP [\sigma_N^2]\right)^2 & = \frac{1}{N^4} \sum_{1 \le i < j \le N} \EbfP \left( \EbfPV [u_N^2(\xi_{ij},V^i,V^j)] - \EbfP [u_N^2(\xi_{ij},V^i,V^j)] \right)^2 \\
		& \le \frac{N(N-1)}{2N^4} \EbfP [u_N^4(\xi_{12},V^1,V^2)] \to 0.
	\end{align*}
	So $\sigma_N^2 \to \sigma^2$ in probability.
	Suppose without loss of generality that $\sigma^2 > 0$, since otherwise we have that $Z = 0$, $U_N \to 0$ in probability as $N \to \infty$ and the desired convergence holds trivially by Theorem \ref{thm:Dynkin}.
	Also note that as $N \to \infty$,
	\begin{equation*}
		\EbfP \sum_{1 \le i < j \le N} \EbfPV \left|\frac{u_N(\xi_{ij},V^i,V^j)}{N}\right|^4 = \frac{N(N-1)}{2N^4} \EbfP [u_N^4(\xi_{12},V^1,V^2)] \to 0.
	\end{equation*}
	Hence the Lyapunov's condition for CLT (see~\cite{Billingsley1995probability}, Theorem 27.3) holds with $\delta=2$:
	\begin{equation*}
		\lim_{N \to \infty} \frac{1}{\sigma_N^{2+\delta}} \sum_{1 \le i < j \le N} \EbfPV \left|\frac{u_N(\xi_{ij},V^i,V^j)}{N}\right|^{2+\delta} = 0,
	\end{equation*}
	where the convergence is in probability.
	It then follows from standard proofs of CLT and a subsequence argument that for each $t \in \Rmb$,
	\begin{equation*}
		\EbfPV \left[e^{itU_N}\right] - e^{-\half t^2 \sigma_N^2} \to 0 
	\end{equation*}
	in probability as $N \to \infty$, which together with the convergence of $\sigma_N^2 \to \sigma^2$ implies that 
	\begin{equation}
		\EbfPV \left[ e^{itU_N} \right] \to e^{-t^2 \sigma^2/2}\label{eq:eq135}
	\end{equation}
	in probability as $N \to \infty$.
	Now let $(t, s_1, \ldots s_K) \mapsto \varphi_N(t,s_1,\dotsc,s_K)$ be the characteristic function of $$(U_N,N^{-\frac{1}{2}} \Umc^N_1 (\phi_1),\dotsc,N^{-\frac{K}{2}} \Umc^N_K (\phi_K)),$$ and $$\varphi(t,s_1,\dotsc,s_K) \doteq e^{-\half t^2\sigma^2} \psi(s_1,\dotsc,s_K),\;\; (t,s_1,\dotsc,s_K)\in \Rmb^{K+1}$$
	 be that of $(Z,I_1(\phi_1),\dotsc,\frac{1}{K!} I_K(\phi_K))$.
	It follows from Theorem \ref{thm:Dynkin} that for all $(s_1,\dotsc,s_K)\in \Rmb^{K}$
	\begin{equation}\EbfP [e^{i \sum_{k=1}^K s_k N^{-\frac{k}{2}} \Umc^N_k (\phi_k)}] \to \psi(s_1,\dotsc,s_K) \mbox{ as }
		N \to \infty.
		\label{eq:eq137}
	\end{equation}
	Thus as $N \to \infty$,
	\begin{align*}
		\varphi_N(t,s_1,\dotsc,s_K) - \varphi(t,s_1,\dotsc,s_K) & = \EbfP \left[ e^{itU_N + i \sum_{k=1}^K s_k N^{-\frac{k}{2}} \Umc^N_k (\phi_k)} - e^{-\half t^2\sigma^2} \psi(s_1,\dotsc,s_K) \right] \\
		& = \EbfP \left[ \left( \EbfPV \left[ e^{itU_N}\right] - e^{-\half t^2\sigma^2} \right) e^{i \sum_{k=1}^K s_k N^{-\frac{k}{2}} \Umc^N_k (\phi_k)} \right] \\
		& \quad + \left(\EbfP [e^{i \sum_{k=1}^K s_k N^{-\frac{k}{2}} \Umc^N_k (\phi_k)}] - \psi(s_1,\dotsc,s_K)\right) e^{-\half t^2\sigma^2} \\
		& \to 0,
	\end{align*}
	where the convergence follows from \eqref{eq:eq135} and \eqref{eq:eq137}.
	This completes the proof.	
\end{proof}
Now we  complete the proof of Proposition \ref{prop:key_joint_cvg}.
From \eqref{eq:JN1til}, \eqref{eq:U_N} and \eqref{eq:h_sym} we can write
\begin{align*} 
	\Jtil^{N,1}(T) &= \sum_{1 \le i < j \le N} \left( \int_0^T \frac{\xi_{ij}(s)-p_N(s)}{N p_N(s)} \bbar_s(X^i_s,X^j_s) \cdot d W^i_s + \int_0^T \frac{\xi_{ij}(s)-p_N(s)}{N p_N(s)} \bbar_s(X^j_s,X^i_s) \cdot d W^j_s \right) \\
	& \quad + \frac{1}{N} \sum_{1 \le i < j \le N} \left( \int_0^T \bbar_s(X^i_s,X^j_s) \cdot d W^i_s + \int_0^T \bbar_s(X^j_s,X^i_s) \cdot d W^j_s \right) \\
	& \quad + \sum_{i=1}^N \int_0^T \frac{1}{N p_N(s)} \bbar_s(X^i_s,X^i_s) \cdot d W^i_s \\
	& = U_N + \frac{2}{N} \Umc^N_2(h^{sym}) + \sum_{i=1}^N \int_0^T \frac{1}{N p_N(s)} \bbar_s(X^i_s,X^i_s) \cdot d W^i_s.
\end{align*}
It follows from Condition \ref{cond:cond2} and law of large numbers that as $N \to \infty$
\begin{equation*}
	\sum_{i=1}^N \int_0^T \frac{1}{N p_N(s)} \bbar_s(X^i_s,X^i_s) \cdot d W^i_s \Rightarrow 0.
\end{equation*}
Also, from  \eqref{eq:eq114}  we have, with $l$ as in \eqref{eq:l},
\begin{equation*}
	\Jtil^{N,2}(T)-\int_0^T \frac{\lambda_s}{p(s)}  \, ds = \sum_{1 \le j \ne k \le N} \frac{N-2}{N^2} \int_0^T m_s(X^j_s,X^k_s) \, ds  = \frac{2(N-2)}{N^2} \Umc_2^N(l).
\end{equation*}
Combining above three displays with \eqref{eq:joint_new}, noting that
$\etatil^N(\phi) =  N^{-\frac{1}{2}} \Umc^N_1 (\phi)$,
and applying Lemma \ref{lem:key} gives us \eqref{eq:joint_cvg}.
This completes the proof of Proposition \ref{prop:key_joint_cvg}. \qed


\bigskip

\section*{Acknowledgement}
SB has been partially supported by NSF-DMS grants 1310002, 160683, 161307 and SES grant 1357622.
AB and RW have been partially supported by the National Science Foundation (DMS-1305120), the Army Research Office (W911NF-14-1-0331) and DARPA (W911NF-15-2-0122).

\bibliographystyle{plain}

\setcounter{equation}{0}
\appendix
\numberwithin{equation}{section}
%
%

\section{Proof of Corollary \ref{cor:LLN}} \label{sec:pf_Cor_LLN}

(a) Fix $\alpha \in \Kbd$ and distinct $i_\alpha, i_\alpha' \in \Nalpha$.
For each fixed $g \in \Cmb_b(\Cmc_d)$ and $x \in \Cmc_d$, let $g^c(x) \doteq g(x) - \langle g, \mu^\alpha \rangle$.
Then as $N \to \infty$
\begin{align*}
	\Ebf \left( \langle g, \mu^{\alpha,N} \rangle - \langle g, \mu^\alpha \rangle \right)^2 & = \Ebf \left(\frac{1}{N_\alpha} \sum_{i \in \Nalpha} g^c(Z^{i,N})\right)^2 \\
	& = \frac{1}{N_\alpha^2} \sum_{i,j \in \Nalpha} \Ebf \left( g^c(Z^{i,N}) g^c(Z^{j,N}) \right) \\
	& = \frac{1}{N_\alpha} \Ebf \left( g^c(Z^{i_\alpha,N}) \right)^2 + \frac{N_\alpha-1}{N_\alpha} \Ebf \left( g^c(Z^{i_\alpha,N}) g^c(Z^{i_\alpha',N}) \right) \\
	& \le \frac{1}{\Nbar} + \Ebf \left( g^c(Z^{i_\alpha,N}) g^c(Z^{i_\alpha',N}) \right),
\end{align*}
which goes to $0$ as $N \to \infty$ by Corollary \ref{cor:poc} and Condition \ref{cond:cond1}. 
This proves part (a).

(b) 
Fix $\alpha, \gamma \in \Kbd$, $i \in \Nmb_\alpha$ and $i_\gamma \in \Nmb_\gamma$.
Abbreviate $\xi_{ij}^N(0), N_{i,\gamma}(0), p_{\alpha\gamma,N}(0)$ as  $\xi_{ij}, N_{i,\gamma}, p_{\alpha\gamma,N}$.
Let $\mubar^{i,\gamma,N} \doteq \frac{1}{N_{i,\gamma}} \sum_{j \in \Ngamma} \xi_{ij} \delta_{X^j} \one_{\{N_{i,\gamma}>0\}}$.
It suffices to show that $d_{BL}(\mu^{i,\gamma,N},\mubar^{i,\gamma,N}) \Rightarrow 0$ and $\mubar^{i,\gamma,N} \Rightarrow \mu^\gamma$ as $N \to \infty$.
Note that
\begin{align*}
	\Ebf \, d_{BL}(\mu^{i,\gamma,N},\mubar^{i,\gamma,N}) & = \Ebf \sup_{\|g||_{BL} \le 1} \left| \langle g, \mu^{i,\gamma,N} \rangle - \langle g, \mubar^{i,\gamma,N} \rangle \right| \\
	& = \Ebf \sup_{\|g||_{BL} \le 1} \left| \sum_{j \in \Ngamma} \frac{\xi_{ij}}{N_{i,\gamma}} \one_{\{N_{i,\gamma}>0\}} \left( g(Z^{j,N}) - g(X^j) \right) \right| \\
	& \le \sum_{j \in \Ngamma} \Ebf \frac{\xi_{ij}}{N_{i,\gamma}} \one_{\{N_{i,\gamma}>0\}} \| Z^{j,N} - X^j \|_{*,T}.
\end{align*}
By similar arguments as in \eqref{eq:TN2_gamma}--\eqref{eq:TN2_temp} we see that the above display can be bounded by
\begin{align*}
	& \left(\Ebf \left( \sum_{k \in \Nalpha} \frac{N_\gamma \xi_{k\igamma}}{N_\alpha N_{k,\gamma}} \one_{\{N_{k,\gamma} > 0\}} - 1 \right)^2 \Ebf \left\|Z^{\igamma,N} - X^\igamma\right\|^2_{*,T}\right)^{1/2} + \Ebf \left\|Z^{\igamma,N} - X^\igamma\right\|_{*,T} \\
	& \quad \le \kappa \left(\frac{1}{N_\alpha p_{\alpha\gamma,N}} + e^{-N_\gamma p_{\alpha\gamma,N}}\right)^{1/2} + \Ebf \left\|Z^{\igamma,N} - X^{\igamma}\right\|_{*,T} \\
	& \quad \le \kappa \left(\frac{1}{\Nbar \pbar_N} + e^{-\Nbar \pbar_N}\right)^{1/2} + \Ebf \left\|Z^{\igamma,N} - X^{\igamma}\right\|_{*,T},
\end{align*}
where the first inequality is from Lemma \ref{lem:prep_2} and \eqref{eq:ZiXi_bd}. From Theorem \ref{thm:NpN_rate} and Condition \ref{cond:cond1}, the last quantity
goes to $0$ as $N \to \infty$.
Thus
$d_{BL}(\mu^{i,\gamma,N},\mubar^{i,\gamma,N}) \Rightarrow 0$ as $N \to \infty$.

Next we show that $\mubar^{i,\gamma,N} \Rightarrow \mu^\gamma$ as $N \to \infty$.
Recall that for $g \in \Cmb_b(\Cmc_d)$ and $x \in \Cmc_d$,  $g^c(x) \doteq g(x) - \langle g, \mu^\gamma \rangle$.
Then as $N \to \infty$,
\begin{align*}
	\Ebf \left( \langle g, \mubar^{i,\gamma,N} \rangle - \langle g, \mu^\gamma \rangle \one_{\{N_{i,\gamma}>0\}} \right)^2 & = \Ebf \left( \sum_{j \in \Ngamma} \frac{\xi_{ij}}{N_{i,\gamma}} \one_{\{N_{i,\gamma}>0\}} g^c(X^j) \right)^2 \\
	& = \Ebf \sum_{j \in \Ngamma} \frac{\xi_{ij}}{N_{i,\gamma}^2} \one_{\{N_{i,\gamma}>0\}} \left( g^c(X^j) \right)^2.
\end{align*}	
The expression on the right can be bounded above by	
\begin{align*} 4 \|g\|_\infty^2 \Ebf \sum_{j \in \Ngamma} \frac{\xi_{ij}}{N_{i,\gamma}^2} \one_{\{N_{i,\gamma}>0\}}  = 4 \|g\|_\infty^2 \Ebf \frac{1}{N_{i,\gamma}} \one_{\{N_{i,\gamma}>0\}}  \le 4 \|g\|_\infty^2 \Ebf \frac{2}{N_{i,\gamma}+1} \le \frac{8\|g\|_\infty^2}{N_\gamma p_{\alpha\gamma,N}}  \le \frac{8\|g\|_\infty^2}{\Nbar \pbar_N}  \to 0,
\end{align*}
where the second inequality follows from Lemma \ref{lem:prep_1}.
Also as $N \to \infty$,
\begin{equation*}
	\Ebf \left( \langle g, \mu^\gamma \rangle \one_{\{N_{i,\gamma}=0\}} \right)^2 \le \|g\|_\infty^2 (1-p_{\alpha\gamma,N})^{N_\gamma} \le \|g\|_\infty^2 e^{-N_\gamma p_{\alpha\gamma,N}} \le \|g\|_\infty^2 e^{-\Nbar\pbar_N} \to 0.
\end{equation*}
Combing above two convergence results implies that $\mubar^{i,\gamma,N} \Rightarrow \mu^\gamma$ as $N \to \infty$, and part (b) follows. \qed

\section{Proof of Lemmas from Section \ref{sec:pre_results}}
\label{sec:pf_section_5}

In this section we give the proofs of Lemmas \ref{lem:prep_1}, \ref{lem:prep_2}, \ref{lem:prep_4}.
\subsection{Proof of Lemma \ref{lem:prep_1}}

	Note that
	\begin{gather*}
		\Ebf \frac{1}{X+1} = \sum_{k=0}^n \frac{1}{k+1} \binom{n}{k} p^k q^{n-k} = \frac{1}{(n+1)p} \sum_{k=0}^n \binom{n+1}{k+1} p^{k+1} q^{n-k} = \frac{1-q^{n+1}}{(n+1)p}.
	\end{gather*}
	Hence for each $m \in \Nmb$,
	\begin{equation*}
		\Ebf \frac{1}{X+m} = \Ebf \frac{1}{X+1} \frac{X+1}{X+m} \le \Ebf \frac{1}{X+1} \frac{n+1}{n+m} = \frac{1-q^{n+1}}{(n+1)p} \frac{n+1}{n+m} = \frac{1-q^{n+1}}{(n+m)p} \le \frac{1}{(n+m)p}.
	\end{equation*}
	Similarly,
	\begin{align*}
	\Ebf \frac{1}{(X+1)^m} & \le \Ebf \frac{m^m}{(X+1)(X+2)\dotsm(X+m)} = \sum_{k=0}^n \frac{m^m}{(k+1)(k+2)\dotsm(k+m)} \binom{n}{k} p^k q^{n-k} \\
	& \le \frac{m^m}{(n+1)(n+2)\dotsm(n+m)p^m} \le \frac{m^m}{(n+1)^m p^m},
	\end{align*}
	which completes the proof.
	\qed

\subsection{Proof of Lemma \ref{lem:prep_2}}

	Fix $\alpha \ne \gamma$, $i_\gamma \in \Ngamma$ and $i_\alpha \in \Nalpha$.
	Since 
	\begin{equation*}
		\Lmc(\zeta_{ki_\gamma}, N_{k,\gamma}) = \Lmc(\zeta_{i_\alpha i_\gamma}, N_{i_\alpha,\gamma}) = \Lmc(\zeta_{i_\alpha j}, N_{i_\alpha,\gamma}), \quad k \in \Nalpha, j \in \Ngamma,
	\end{equation*} 
	we have 
	\begin{align*}
		\Ebf \sum_{k \in \Nalpha} \frac{N_\gamma \zeta_{ki_\gamma}}{N_\alpha N_{k,\gamma}} \one_{\{N_{k,\gamma} > 0\}} = \Ebf \frac{N_\gamma \zeta_{i_\alpha i_\gamma}}{N_{i_\alpha,\gamma}} \one_{\{N_{i_\alpha,\gamma} > 0\}} = \Ebf \sum_{k \in \Ngamma} \frac{\zeta_{i_\alpha k}}{N_{i_\alpha,\gamma}} \one_{\{N_{i_\alpha,\gamma} > 0\}} = \Pbd(N_{i_\alpha,\gamma} > 0).
	\end{align*}
	Using this we can write
	\begin{align}
		& \Ebf \left( \sum_{k \in \Nalpha} \frac{N_\gamma \zeta_{ki_\gamma}}{N_\alpha N_{k,\gamma}} \one_{\{N_{k,\gamma} > 0\}} - 1 \right)^2 \notag \\
		& = \frac{N_\gamma^2}{N_\alpha^2} \Ebf \sum_{k,l \in \Nalpha} \frac{\zeta_{ki_\gamma} \zeta_{li_\gamma}}{N_{k,\gamma} N_{l,\gamma}} \one_{\{N_{k,\gamma} > 0\}} \one_{\{N_{l,\gamma} > 0\}} - 2 \Pbd(N_{i_\alpha,\gamma} > 0) + 1 \notag \\
		& = \frac{N_\gamma^2}{N_\alpha^2} \Ebf \sum_{k \in \Nalpha} \frac{\zeta_{ki_\gamma}}{N_{k,\gamma}^2} \one_{\{N_{k,\gamma} > 0\}} + \frac{N_\gamma^2}{N_\alpha^2} \Ebf \sum_{k,l \in \Nalpha, k \ne l} \frac{\zeta_{ki_\gamma} \zeta_{li_\gamma}}{N_{k,\gamma} N_{l,\gamma}} \one_{\{N_{k,\gamma} > 0\}} \one_{\{N_{l,\gamma} > 0\}} + 2 \Pbd(N_{i_\alpha,\gamma} = 0) - 1 \notag \\
		& \doteq \sum_{n=1}^3 \Tmc^{N,n}_\gamma - 1, \label{eq:pf_prep3_key}
	\end{align}
	For $\Tmc^{N,1}_\gamma$, we have  by a straightforward conditioning argument,
	\begin{align}
		\Tmc^{N,1}_\gamma & \doteq \frac{N_\gamma^2}{N_\alpha^2} \sum_{k \in \Nalpha} \Ebf \frac{\zeta_{ki_\gamma}}{N_{k,\gamma}^2} \one_{\{N_{k,\gamma} > 0\}} = \frac{N_\gamma^2}{N_\alpha^2} \sum_{k \in \Nalpha} p_{\alpha\gamma,N} \Ebf \frac{1}{(N_{k,\gamma}-\zeta_{ki_\gamma}+1)^2} \notag \\
		& \le \frac{N_\gamma^2}{N_\alpha^2} N_\alpha p_{\alpha\gamma,N} \frac{4}{(N_\gamma p_{\alpha\gamma,N})^2} = \frac{4}{N_\alpha p_{\alpha\gamma,N}}, \label{eq:pf_prep3_bound1}
	\end{align}
	where the inequality follows from Lemma \ref{lem:prep_1}.
	For $\Tmc^{N,2}_\gamma$, using the independence of $(\zeta_{ki_\gamma}, N_{k,\gamma})$ and $(\zeta_{li_\gamma}, N_{l,\gamma})$ for different $k,l \in \Nalpha$, we have
	\begin{align}
		\Tmc^{N,2}_\gamma 
		& = \frac{N_\gamma^2}{N_\alpha^2} \sum_{k,l \in \Nalpha, k \ne l} \Ebf \left( \frac{\zeta_{ki_\gamma}}{N_{k,\gamma}} \one_{\{N_{k,\gamma} > 0\}} \right) \Ebf \left( \frac{\zeta_{li_\gamma}}{N_{l,\gamma}} \one_{\{N_{l,\gamma} > 0\}} \right) \notag \\
		& = \frac{N_\gamma^2}{N_\alpha^2} \sum_{k,l \in \Nalpha, k \ne l} p_{\alpha\gamma,N}^2 \Ebf \left( \frac{1}{N_{k,\gamma}-\zeta_{ki_\gamma}+1} \right) \Ebf \left( \frac{1}{N_{l,\gamma}-\zeta_{li_\gamma}+1} \right) \notag \\
		& \le \frac{N_\gamma^2}{N_\alpha^2} N_\alpha (N_\alpha-1) p_{\alpha\gamma,N}^2 \left( \frac{1}{N_\gamma p_{\alpha\gamma,N}} \right)^2 \le 1, \label{eq:pf_prep3_bound2}
	\end{align}
	where the first inequality once more follows from Lemma \ref{lem:prep_1}.
	Finally, for $\Tmc^{N,3}_\gamma$, we have
	\begin{equation}
		\label{eq:pf_prep3_bound3}
		\Tmc^{N,3}_\gamma \doteq 2 \Pbd(N_{i_\alpha,\gamma} = 0) = 2 (1 - p_{\alpha\gamma,N})^{N_\gamma} \le 2 e^{-N_\gamma p_{\alpha\gamma,N}}.
	\end{equation}		
	Plugging \eqref{eq:pf_prep3_bound1}--\eqref{eq:pf_prep3_bound3} into \eqref{eq:pf_prep3_key} gives the first statement when $\alpha \neq \gamma$. The case $\alpha =\gamma$ is immediate
	from the second statement which we prove next.

	Fix $i_\alpha \in \Nalpha$.
	Since
	\begin{equation}
		\label{eq:pf_prep3_exchange_new}
		\Lmc(\zeta_{ki_\alpha}, N_{k,\alpha}) = \Lmc(\zeta_{i_\alpha k}, N_{i_\alpha,\alpha}), \quad k \in \Nalpha,
	\end{equation} 
	we have
	\begin{equation*}
		\Ebf \sum_{k \in \Nalpha} \frac{\zeta_{ki_\alpha}}{N_{k,\alpha}} = \Ebf \sum_{k \in \Nalpha} \frac{\zeta_{i_\alpha k}}{N_{i_\alpha,\alpha}} = 1.
	\end{equation*}
	Using this we can write
	\begin{align}
		\Ebf \left( \sum_{k \in \Nalpha} \frac{\zeta_{ki_\alpha}}{N_{k,\alpha}} - 1 \right)^2 & = \Ebf \left( \sum_{k \in \Nalpha} \frac{\zeta_{ki_\alpha}}{N_{k,\alpha}} \right)^2 - 2 \Ebf \sum_{k \in \Nalpha} \frac{\zeta_{ki_\alpha}}{N_{k,\alpha}} + 1 = \Ebf \sum_{k,l \in \Nalpha} \frac{\zeta_{ki_\alpha} \zeta_{li_\alpha}}{N_{k,\alpha} N_{l,\alpha}} - 1 \notag \\
		& = \sum_{n=1}^4 \Ebf \sum_{(k,l) \in \Smc^{N,n}_\alpha} \frac{\zeta_{ki_\alpha} \zeta_{li_\alpha}}{N_{k,\alpha} N_{l,\alpha}} - 1 \doteq \sum_{n=1}^4 \Tmc^{N,n}_\alpha - 1, \label{eq:pf_prep3_key_new}
	\end{align}
	where $\Smc^{N,1}_\alpha$, $\Smc^{N,2}_\alpha$, $\Smc^{N,3}_\alpha$ and $\Smc^{N,4}_\alpha$ are collections of $(k,l) \in \Nalpha \times \Nalpha$ such that $\{ k = l \}$, $\{k \ne l, k=i_\alpha \}$, $\{ k \ne l, l=i_\alpha \}$ and $\{ i_\alpha,k,l \text{ distinct} \}$, respectively.
	For $\Tmc^{N,1}_\alpha$, it follows from \eqref{eq:pf_prep3_exchange_new} that
	\begin{equation} \label{eq:pf_prep3_bound1_new}
		\Tmc^{N,1}_\alpha \doteq \sum_{k \in \Nalpha} \Ebf \frac{\zeta_{ki_\alpha}}{N_{k,\alpha}^2} = \sum_{k \in \Nalpha} \Ebf \frac{\zeta_{i_\alpha k}}{N_{i_\alpha,\alpha}^2} = \Ebf \frac{1}{N_{i_\alpha,\alpha}} \le \frac{1}{N_\alpha p_{\alpha\alpha,N}},
	\end{equation}
	where the inequality follows from Lemma \ref{lem:prep_1}.
	For $\Tmc^{N,2}_\alpha$, using independence of $N_{i_\alpha,\alpha} - \zeta_{li_\alpha}$ and $N_{l,\alpha} - \zeta_{li_\alpha}$ for different $i_\alpha, l \in \Nalpha$, we have
	\begin{align}
		\Tmc^{N,2}_\alpha & = \sum_{l \in \Nalpha, l \ne i_\alpha} \Ebf \frac{\zeta_{li_\alpha}}{N_{i_\alpha,\alpha} N_{l,\alpha}} = \sum_{l \in \Nalpha, l \ne i_\alpha} p_{\alpha\alpha,N} \Ebf \frac{1}{(N_{i_\alpha,\alpha} + 1 - \zeta_{li_\alpha})(N_{l,\alpha} + 1 - \zeta_{li_\alpha})} \notag \\
		& = \sum_{l \in \Nalpha, l \ne i_\alpha} p_{\alpha\alpha,N} \Ebf \frac{1}{N_{i_\alpha,\alpha} + 1 - \zeta_{li_\alpha}} \Ebf \frac{1}{N_{l,\alpha} + 1 - \zeta_{li_\alpha}} \notag \\
		& \le (N_\alpha - 1) p_{\alpha\alpha,N} \frac{1}{(N_\alpha p_{\alpha\alpha,N})^2} \le \frac{1}{N_\alpha p_{\alpha\alpha,N}}, \label{eq:pf_prep3_bound2_new}
	\end{align}
	where the first inequality again follows from Lemma \ref{lem:prep_1}.
	Similarly for $\Tmc^{N,3}_\alpha$, we have
	\begin{equation} 
		\label{eq:pf_prep3_bound3_new}
		\Tmc^{N,3}_\alpha = \sum_{k \in \Nalpha, k \ne i_\alpha} \Ebf \frac{\zeta_{ki_\alpha}}{N_{i_\alpha,\alpha} N_{k,\alpha}} \le \frac{1}{N_\alpha p_{\alpha\alpha,N}}.
	\end{equation}
	Finally for $\Tmc^{N,4}_\alpha$, using the independence of $N_{k,\alpha}-\zeta_{ki_\alpha}-\zeta_{kl}$ and $N_{l,\alpha}-\zeta_{li_\alpha}-\zeta_{kl}$ for distinct $i_\alpha,k,l \in \Nalpha$, we have
	\begin{align}
		\Tmc^{N,4}_\alpha & = \sum_{(k,l) \in \Smc^{N,4}_\alpha} \Ebf \frac{\zeta_{ki_\alpha} \zeta_{li_\alpha}}{N_{k,\alpha} N_{l,\alpha}} = \sum_{(k,l) \in \Smc^{N,4}_\alpha} p_{\alpha\alpha,N}^2 \Ebf \frac{1}{(N_{k,\alpha}+1-\zeta_{ki_\alpha})(N_{l,\alpha}+1-\zeta_{li_\alpha})} \notag \\
		& \le \sum_{(k,l) \in \Smc^{N,4}_\alpha} p_{\alpha\alpha,N}^2 \Ebf \frac{1}{(N_{k,\alpha}+1-\zeta_{ki_\alpha}-\zeta_{kl})(N_{l,\alpha}+1-\zeta_{li_\alpha}-\zeta_{kl})} \notag \\
		& = \sum_{(k,l) \in \Smc^{N,4}_\alpha} p_{\alpha\alpha,N}^2 \Ebf \frac{1}{N_{k,\alpha}+1-\zeta_{ki_\alpha}-\zeta_{kl}} \Ebf \frac{1}{N_{l,\alpha}+1-\zeta_{li_\alpha}-\zeta_{kl}} \notag \\
		& \le (N_\alpha-1)(N_\alpha-2) p_{\alpha\alpha,N}^2 \left( \frac{1}{(N_\alpha-1) p_{\alpha\alpha,N}} \right)^2 \le 1, \label{eq:pf_prep3_bound4_new}
	\end{align}
	where the second inequality once more follows from Lemma \ref{lem:prep_1}.
	Plugging \eqref{eq:pf_prep3_bound1_new}--\eqref{eq:pf_prep3_bound4_new} into \eqref{eq:pf_prep3_key_new} completes the proof.
	\qed

\subsection{Proof of Lemma \ref{lem:prep_4}}

	First note that the result holds trivially when $p_N = 0$ or $p_N = 1$.
	Now consider the case $p_N \in (0,1)$.
	For $t \ge 0$, it follows from Hoeffding's inequality that
	\begin{equation*}
		\Pbd \left( |Y - Np_N| > t + 1 \right) \le \Pbd \left( \left| \sum_{i=2}^N (\zeta_i - p_N) \right| > t \right) \le 2 e^{- \frac{2t^2}{N-1}}.
	\end{equation*}
	Taking $t = C_N(k)$ completes the proof.
	\qed

\section{A lemma on integral operators} \label{sec:restating}

Let $\Smb$ be a Polish space and $\nu \in \Pmc(\Smb)$.
Let $a(\cdot,\cdot) \in L^2(\nu \otimes \nu)$ and denote by $A$ the integral operator on $L^2(\nu)$ associated with $a$: $A \phi(x) \doteq \int_\Smb a(x,y) \phi(y) \, \nu(dy)$ for $x \in \Smb$ and $\phi \in L^2(\nu)$.
Then $A$ is a Hilbert-Schmidt operator.
Also, $AA^*$, and for $n \ge 2$, $A^n$, are trace class operators.
The following lemma is taken from  Shiga-Tanaka \cite{ShigaTanaka1985}.

\begin{lemma} \label{lem:Shiga_Tanaka}
	Suppose that \textnormal{Trace}$(A^n)=0$ for all $n \ge 2$.
	Then $\Ebf [e^{\frac{1}{2} I_2(f)}] = e^{\frac{1}{2} \textnormal{Trace} (AA^*)}$, where $f(x,y) \doteq a(x,y) + a(y,x) - \int_\Smb a(x,z) a(y,z) \, \nu(dz)$, and $I_2(\cdot)$ is the MWI defined as in Section \ref{sec:asymp_symmetric_statistics}.
	Moreover, $I-A$ is invertible and for any $\phi \in L^2(\nu)$, 
	$$\Ebf \left[\exp(iI_1(\phi) + \frac{1}{2} I_2(f))\right] = \exp\left\{-\frac{1}{2} (\| (I-A)^{-1} \phi \|^2_{L^2(\nu)} - \textnormal{Trace} (AA^*))\right\},$$
	 where $I$ is the identity operator on $L^2(\nu)$.
\end{lemma}


%

\end{document}